\documentclass[12pt]{amsart}
\usepackage{enumerate}
\usepackage{amssymb}
\usepackage{graphicx}
\usepackage{amscd, color}
\usepackage{amsmath}
\usepackage{amsfonts}
\usepackage[english]{babel}
\usepackage{epsfig}
\usepackage{latexsym,amssymb,amsmath}
\usepackage{mathrsfs}
\usepackage{hyperref}
\usepackage{amsthm}
\usepackage{soul}

\usepackage{float}
\usepackage[T1]{fontenc}

\usepackage{fancyvrb}
\usepackage{alltt}
\usepackage{fancyhdr}
\usepackage[parfill]{parskip}
\usepackage[shortlabels]{enumitem}
\usepackage{afterpage}
\usepackage{setspace}

\usepackage[all,cmtip]{xy}
\usepackage[bottom]{footmisc}

\usepackage{xr}
\usepackage{tikz} 
\usetikzlibrary{matrix}

\numberwithin{equation}{section}
\newtheorem{theorem}{\textit{Theorem}}[section]
\newtheorem{proposition}[theorem]{\textit{Proposition}}
\newtheorem{definition}[theorem]{\textit{Definition}}
\newtheorem{lemma}[theorem]{\textit{Lemma}}
\newtheorem{corollary}[theorem]{\textit{Corollary}}

\newtheorem{remark}[theorem]{\textit{Remark}}

\newtheorem{theoremm}{\textit{Theorem}}
\newtheorem*{theoremmA}{\textit{Theorem~A}}
\newtheorem{corollaryy}[theoremm]{\textit{Corollary}}
\newtheorem{fact}{\textbf{Fact}}
\newcommand*{\QEDA}{\hfill\ensuremath{\blacksquare}}
\DeclareMathOperator{\Int}{Int}
\DeclareMathOperator{\cco}{\overline{conv}}

\title[The topology of the polar involution]{A topological insight into the polar involution of convex sets}
\author{Luisa F. Higueras-Monta{\~n}o and Natalia Jonard-P{\'e}rez}

\address{Departamento de  Matem\'aticas,
Facultad de Ciencias, Universidad Nacional Aut\'onoma de M\'exico, 04510 Ciudad de M\'exico, M\'exico.}

\email{(L.\,Higueras-Monta\~no) fher@cimat.mx}
\email{(N.\,Jonard-P\'erez) nat@ciencias.unam.mx}



\thanks{
The first author has been supported by The Post-Doctoral Scholarship Program at UNAM. The second author have been supported by CONACyT grant 252849 (M\'exico) and by PAPIIT grant IN101622  (UNAM, M\'exico).}

\begin{document}


\textit{This paper has been accepted for publication in Israel Journal of \mbox{Mathematics}}

\vspace{0.7cm}

{\let\newpage\relax\maketitle}

\begin{abstract}
  Denote by $\mathcal{K}_0^n$ the family of all closed convex sets $A\subset\mathbb{R}^n$ containing the origin $0\in\mathbb R^n$. For $A\in\mathcal{K}_0^n,$ its polar set is denoted by $A^\circ.$ In this paper, we investigate the topological nature of the polar mapping $A\to A^\circ$ on $(\mathcal{K}_0^n, d_{AW})$, where $d_{AW}$ denotes the Attouch-Wets metric. We prove that  $(\mathcal{K}_0^n, d_{AW})$ is homeomorphic to the Hilbert cube $Q=\prod_{i=1}^{\infty}[-1,1]$ and the polar mapping is topologically conjugate with the standard based-free involution $\sigma:Q\rightarrow Q,$ defined by $\sigma(x)=-x$ for all $x\in Q.$ 
We also prove that among the inclusion-reversing involutions on $\mathcal K^n_0$ (also called dualities), those and only those with a unique fixed point are topologically conjugate with the polar mapping, and they can be characterized as all the maps $f:\mathcal{K}_0^n\to \mathcal{K}_0^n$ of the form $f(A)=T(A^{\circ})$, with $T$  a positive-definite linear isomorphism of $\mathbb R^n$.  
\end{abstract}




\section{Introduction}
\label{sec:intro}

Let $\mathbb{R}^n,$ $n\geq2,$ be  the $n$-dimensional Euclidean space endowed with the standard scalar product $\langle\cdot,\cdot\rangle$, and denote by $\mathcal{K}_0^n$  the class of closed convex sets $A\subset\mathbb{R}^n$ such that $0\in A$. The \textit{polar set} of a given element $A\in\mathcal{K}_0^n$ is defined as 
\begin{equation}
\label{eq:polarsetdefn}
A^\circ:=\left\{x\in\mathbb{R}^n:\sup_{a\in A}\langle a,x\rangle\leq1\right\},
\end{equation} 
and it is always an element of  $\mathcal{K}_0^n$. A classic result related with the polar set establishes  that $\big(A^\circ\big)^\circ=A$ for all $A\in\mathcal{K}_0^n$ (see the Bipolar Theorem in Section~\ref{sec:prelim}). 
This property leads us to the well-known concept of involution.
An \textit{involution}  on a topological space $X$ is a continuous map $\beta:X\rightarrow X$ such that  $\beta(\beta(x))=x$ for every $x\in X$. If additionally $\beta$ has a unique fixed point (namely, if there exists a unique point $x_0\in X$ such that $\beta(x_0)=x_0$), then $\beta$ is called \textit{based-free} (c.f. Definition~\ref{def:base free action}).

Let $X$ and $Y$ be topological spaces with involutions $\beta$ and $\tau$, respectively. A map $f:X\to Y$ is called \textit{equivariant}  (with respect to $\beta$ and $\tau$), if $f(\beta (x))=\tau (f(x))$ for every $x\in X$. 
If we can find an equivariant homeomorphism $f:X\to Y$, then we say that $\beta$ and $\tau$ are \textit{conjugate}.

An interesting and difficult problem in infinite-dimensional topology,  is the characterization of all based-free involutions on the Hilbert cube $Q=\prod_{i=1}^{\infty}[-1,1]$. In the mid-sixties, R. D. Anderson  asked if all based-free involutions on $Q$ are conjugate (see e.g.\cite[Problem 930]{West1990} or  \cite[Section 3.2]{ANTONYAN2021}). In other words, if we let $\sigma:Q\to Q$  be the standard involution on $Q$ given by $\sigma(x)=-x$, $x\in Q,$ Anderson's problem can be stated as:

\begin{center}\textit{If $\beta:Q\rightarrow Q$ is an involution with a unique fixed point,  does there exist a homeomorphism $\Psi:Q\rightarrow Q$ such that $\beta=\Psi\sigma\Psi^{-1}$? }
\end{center}
(see \cite[p 554]{West1990}).  Despite the many efforts that have been done to answer this question (see for instance \cite{Antonyan1999,vanMillWest2020, vanMillWest2022,West1976,west2018,WestWong79,Wong1974}), Anderson's problem remains open. 

An important tool in 
this and other topics on infinite-dimensional topology has been the use of different models for the space $Q$.  In this respect, the so called hyperspaces of compact sets have played an important role (see, e.g., \cite{CurtisSchori74}, \cite{CurtisSchori78} and  \cite{west2018}). In particular, in the seminal work \cite{nadler1979} the authors proved that the family of all compact convex subsets of a convex set $K\subset\mathbb R^n$ is homemorphic to the Hilbert cube, provided that $K$ is compact and has dimension at least 2.

In this work, we are interested in  the topological properties of the polar mapping and its relationship with the so called Anderson's problem. In this sense, the 
set $\mathcal{K}_0^n$ can be considered as a metric space with respect to the  Attouch-Wets metric $d_{AW}$ (see Section \ref{sec:prelim} for the definition). In this case, the polar mapping 
\begin{align*}
\alpha:(\mathcal{K}_0^n,d_{AW})&\rightarrow(\mathcal{K}_0^n,d_{AW})\\
A&\to A^\circ
\end{align*}  
is in fact a based-free  involution for which the Euclidean ball $\mathbb{B}$ is the unique fixed point (see e.g. Proposition \ref{prop:polar-continuous}). Additionally, as we will prove in Section~\ref{sec:K0n Hilbert cube}, the space $(\mathcal{K}_0^n,d_{AW})$ is homeomorphic to the Hilbert cube $Q.$ 
One of our main contributions is the following  theorem.

\begin{theoremm}\label{Theo:main 1}
The polar mapping $\alpha:(\mathcal{K}_0^n,d_{AW})\rightarrow(\mathcal{K}_0^n,d_{AW})$ is topologically conjugate with the standard involution $\sigma:Q\rightarrow Q$.
\end{theoremm} 

It is worth mentioning that the homeomorphism $\Psi:Q\to\mathcal{K}_0^n$ obtained after Theorem~\ref{Theo:main 1}, does not map the natural order on $Q$ (see Section~\ref{sec:remarks}) into the order inclusion on $\mathcal K_0^n$ (Remark~\ref{rem: order}).

In \cite{ArtsteinMilman2007,ArtsteinMilman2008,SchneiderBoroczky2008,Slomka2011}, a deep study on the geometric properties of the polar mapping 
has been carried out . It is known that this map is a duality (in the sense of \cite[Definition 4]{ArtsteinMilman2007}) for the class $\mathcal{K}_0^n$, and it can be characterized by some of its basic properties.
For example, in \cite[Corollary 4]{Slomka2011}  (see also \cite[Theorem 10]{ArtsteinMilman2008}),
the following is proved:
\begin{theoremmA}
\label{thm:Slomka2011-corolario}
\cite[Corollary 4]{Slomka2011}
Let $n\geq2$. Let $f:\mathcal{K}_0^n\rightarrow\mathcal{K}_0^n$ be a mapping satisfying for all $A,K\in\mathcal{K}_0^n$ that
\begin{itemize}
    \item[(D1)] $f(f(A))=A$, 
\item[(D2)] If $A\subseteq K$, then $f(A)\supseteq f(K)$.
\end{itemize}
Then, there exists a symmetric linear isomorphism $T:\mathbb{R}^n\rightarrow\mathbb{R}^n$ such that
$f(A)=T(A^\circ)$ for all $A\in\mathcal{K}_0^n$.
\end{theoremmA}
An earlier version of this result was proved in \cite{SchneiderBoroczky2008} for the class $\mathcal{K}_{(0),b}^n$ of convex compact sets containing $0$ in their interior. 

We will notice in Section~\ref{sec:other standard involutions}  that all the maps $f:\mathcal{K}_0^n\rightarrow\mathcal{K}_0^n$ satisfying conditions (D1) and (D2)  
are always continuous on $(\mathcal{K}_0^n,d_{AW})$ (see Remark~\ref{rem:continuity-decreasing-inv}).
For short, these kinds of maps will be called \textit{decreasing involutions}. 
A key observation
is that Theorem~A
translates the whole study of decreasing involutions into the linear category. In this setting,
another of our results will be the characterization of the linear isomorphisms on $\mathbb{R}^n$  that induce  based-free involutions.
More precisely, these mappings can be characterized as follows.

\begin{theoremm}
\label{thm:caracterizacion-decreasing}
Let $T:\mathbb{R}^n\rightarrow\mathbb{R}^n$ be a symmetric linear isomorphism and let $f:\mathcal{K}_0^n\rightarrow\mathcal{K}_0^n$ be defined as $f(A)=T(A^\circ)$. Then, the following statements hold.
\begin{enumerate}
    \item  If $T$ is positive-definite, then $f$ is conjugate with the polar mapping. In particular, $f$ is based-free.
    \item If $T$ is not positive-definite, then $f$ has infinitely many fixed points in $\mathcal{K}_{(0),b}^n$.
\end{enumerate}
\end{theoremm}

Theorem~\ref{thm:caracterizacion-decreasing}, in combination with Theorem \ref{Theo:main 1} and Theorem~A, 
yields the following corollary that could be of interest given its  relation with Anderson's Problem. 

\begin{corollaryy}
\label{cor:answer}
Every based-free decreasing involution  $f:\mathcal K_0^n\to\mathcal K_0^n$ is conjugate with the standard involution on $Q$. Moreover, $f$ is of the form  $f(A)=T(A^\circ)$ for some positive-definite linear isomorphism $T:\mathbb{R}^n\rightarrow\mathbb{R}^n.$ 
\end{corollaryy}

In terms of \cite[Definition 4]{ArtsteinMilman2007}, one can say that a duality on $\mathcal{K}_0^n$  is topologically equivalent to the polar mapping if and only if  it has a unique fixed point.

Let us briefly describe the contents and structure of the paper. In Section \ref{sec:prelim}, we introduce the notation and basic results  regarding $\mathcal K^n_0$, the Attouch-Wets metric and  the polar set. We also introduce some tools from the theory of Hilbert cube manifolds, $\rm{ANR}$'s and $G$-spaces that will be used later. Section \ref{sec:K0n Hilbert cube} is dedicated to calculate the topological structure of the space $\mathcal{K}_0^n,$ $n\geq2.$ In Theorem  \ref{thm:K0cuboHilbert}, we  show that $(\mathcal{K}_0^n,d_{AW})$  is homeomorphic to $Q$. Additionally, in Corollary \ref{cor:topo-strucK0b}, we prove that the subspace $\mathcal{K}_{0,b}^n$ consisting of all compact convex sets containing $0$ is homeomorphic to the Hilbert cube with a point removed. 

Section \ref{sec:K0n Z2-AR} is dedicated to prove Theorem~\ref{Theo:main 1}. We begin by summarizing the basic properties of the polar mapping $\alpha:\mathcal{K}_0^n\rightarrow\mathcal{K}_0^n$ and the $\mathbb Z_2$-action induced by it.  After recalling some results from the theory of $G$-spaces,  we prove that the $\mathbb{Z}_2$-space  $(\mathcal{K}_{0}^n,\alpha)$ is  $\mathbb{Z}_2$-contractible (Theorem \ref{thm:K0-Z2-contractible}). This allows us to prove that the polar mapping is conjugate to $\sigma$ (Theorem~\ref{Theo:main 1}), among other results (Corollary~\ref{cor:K0-AR}). 

In Section \ref{sec:polarpreservingmap}, we investigate the properties of the maps on $\mathbb{R}^n$  that preserve polarity. A map $f:\mathbb{R}^n\rightarrow\mathbb{R}^n$ is called \textit{polar preserving map}   if  $f(A^\circ)=f(A)^\circ$ for all $A\in\mathcal{K}_{0}^n.$ Proposition \ref{prop:injective-polarpreservingmap} is the main result of this section. There, we prove that every injective polar preserving map $f:\mathbb{R}^n\rightarrow\mathbb{R}^n$ must be an orthogonal map.

In Section~\ref{sec:other standard involutions}, 
we characterize based-free decreasing involutions on $\mathcal K^n_0$. Namely, we prove Theorem~\ref{thm:caracterizacion-decreasing}. We finish the paper with some remarks and questions that could be of general interest (Section~\ref{sec:remarks}).

\section{Preliminaries}
\label{sec:prelim}

We begin by introducing the notation and basic results we use throughout the work. The Euclidean space $\mathbb{R}^n,$ $n\geq2,$ is endowed with the standard scalar product $\langle\cdot,\cdot\rangle.$ The corresponding Euclidean norm and closed unit ball are denoted by $\|\cdot\|$ and $\mathbb{B},$ respectively. The topology on $\mathbb{R}^n$ is the one determined by $\|\cdot\|.$ 
If $A\subset\mathbb R^n$ is an arbitrary subset, the \textit{closed convex hull} of $A$ is denoted by $\cco(A)$.
For every pair of points $a,b\in\mathbb R^n$, we denote by $[a,b]$ the segment with end points $a$ and $b$. Namely,
$$[a,b]:=\{(1-t)a+tb: t\in [0,1]\}.$$

By $\mathcal{K}^n$ we denote the family of all nonempty closed and convex subsets of $\mathbb R^n$.
The family of all elements in $\mathcal{K}^n$ containing the origin $0\in\mathbb R^n$ will be denoted by $\mathcal{K}_0^n$, and the elements of $\mathcal{K}_0^n$ containing the origin in its interior will be denoted by $\mathcal{K}_{(0)}^n$. We also use the symbol $\mathcal{K}_{b}^n$ to denote the  elements of  $\mathcal{K}^n$ that are compact, and by following this logic,  we   define the families
$$\mathcal{K}_{0,b}^n:=\mathcal{K}_{0}^n\cap \mathcal{K}_{b}^n\;\text{ and }\;\mathcal{K}_{(0),b}^n:=\mathcal{K}_{(0)}^n\cap \mathcal{K}_{b}^n.$$

  On $\mathcal{K}^n$ one can consider different  topologies  (see, e.g. \cite{Beer1993}). In this work, we are interested in the Attouch-Wets topology $\tau_{AW}.$  This topology is determined by the Attouch-Wets metric $d_{AW}$ which can be  defined for $A,K\in\mathcal{K}^n$ as
\begin{equation}
\label{eq:dAW}
d_{AW}(A,K):=\sup_{j\in\mathbb{N}}\left\{\min\left\{\frac{1}{j},\sup_{\|x\|<j}|d(x,A)-d(x,K)|\right\}\right\}.
\end{equation}
This definition is the one used in \cite{SakaiYaguchi2006} and  it is convenient to highlight that the formula in (\ref{eq:dAW}) is equivalent to the Attouch-Wets metric defined in \cite[Definition 3.1.2]{Beer1993}.  

Throughout the work we will use several times the following fact, which can be consulted in \cite[Theorem 3.1.4 and Exercise 5.1.10(b)]{Beer1993} or \cite[Remark 1]{SakaiYang2007}. 
\begin{fact}\label{fact:topologies are the same}
On $\mathcal{K}^n,$ the Attouch-Wets topology,   the Fell topology $\tau_F$ and the Wijsman topology $\tau_W$ are the same.
\end{fact}

For the convenience of the reader, we recall the definition of $\tau_F$ and $\tau_W$. \textit{The Fell topology} on $\mathcal{K}^n$ is generated by the sets $U^-:=\{A\in\mathcal{K}^n: A\cap U\neq\emptyset\}$ and $(\mathbb{R}^n\setminus C)^+:=\{A\in\mathcal{K}^n:A\subset\mathbb{R}^n\setminus C\},$ where $U\subset\mathbb R^n$ is open and $C\subset\mathbb{R}^n$ is compact. This topology is also defined on $\mathcal{K}_{*}^n:=\mathcal{K}^n\cup\{\emptyset\}.$
On the other hand, the Wijsman topology on $\mathcal K^n$ is the one generated by the sets 
$$U^-(x,r):=\{A\in \mathcal{K}^n:d(x,A)<r\}$$
$$U^+(x,r):=\{A\in \mathcal{K}^n:d(x,A)>r\}$$
where $x\in \mathbb R^n$ and $r>0$.

Recall that the \textit{ Hausdorff metric} between any  pair of nonempty  sets $A, K\subset\mathbb R^n$ is defined by any of the following equivalent expressions
\begin{align*}
d_H(A,K)&=\max\left\{\sup_{x\in A}d(x,K), \sup_{x\in K}d(x,A)\right\}\\
&=\text{inf}\left\{\lambda>0:A\subseteq K+\lambda\mathbb{B},\ K\subseteq A+\lambda\mathbb{B}\right\}\\
&=\sup_{x\in\mathbb R^n}|d(x, A)-d(x,K)|
\end{align*}
(see, e.g. \cite[\S 3.2]{Beer1993}). 

It is important to point out the following fact.
\begin{fact}\label{fact:d_H y d_AW coincinden en K_b}
(\cite[Theorem 3.2]{SakaiYaguchi2006})
On the class of compact convex subsets $\mathcal K^n_b$,  the Attouch-Wets metric and the Hausdorff metric $d_H$ determine the same topology.
\end{fact}

In \cite[Lemma 2.3]{DonJuanJonardPerezLopezPoo}, the following useful property of $d_{AW}$  was established: Let $\varepsilon>0$ be such that $\frac{1}{j+1}<\varepsilon\leq\frac{1}{j}$ for some $j\in\mathbb{N},$ and let $A,K\in\mathcal{K}^n.$ Then,
\begin{equation}
\label{eq:Lemma-Ananda}
d_{AW}(A,K)<\varepsilon\textit{ if and only if }\sup_{\|x\|<j}|d(x,A)-d(x,K)|<\varepsilon.
\end{equation}
This relation holds not only for elements of $\mathcal K^n$, but for any pair of nonempty closed subsets. However, in the context we are working on, we can use the Euclidean structure of $\mathbb R^n$ to provide a sharper relationship between $d_{AW}$ and $d_H$ on $\mathcal K^n_0$.

\begin{lemma}
\label{lem:car-dAW}
Let $r>0$. For any $z\in r\mathbb B$ and $A, K\in\mathcal K^n_0$, the following equalities hold. 
\begin{enumerate}
\item  $d(z,A)=d(z, A\cap r\mathbb B)$.
\item $\sup\limits_{\|x\|<r}|d(x,A)-d(x,K)|=\sup\limits_{x\in r\mathbb B}|d(x,A)-d(x,K)|.$
\item $d_H\big(A\cap r\mathbb B,K\cap r\mathbb B\big)=\sup\limits_{x\in r\mathbb B}|d(x,A)-d(x,K)|$.  

\item $d_{AW}(A,K)=\sup\limits_{j\in\mathbb{N}}\left\{\min\left\{\frac{1}{j},d_H\big(A\cap j\mathbb B,K\cap j\mathbb B\big)\right\}\right\}.$  

\item For every integer $j\geq 1$ and every $\varepsilon \in \left(\frac{1}{j+1}, \frac{1}{j}\right]$,
$$d_{AW}(A,K)<\varepsilon\textit{ if and only if }d_H\big(A\cap j\mathbb B,K\cap j\mathbb B\big)<\varepsilon.$$
\end{enumerate}
\end{lemma}
\begin{proof}

(1) Let $a\in A$ be the unique element such that $d(z,A)=\|z-a\|$. 
Assume that   $\|a\|>r$ and consider the point $p:=\frac{\langle z,a\rangle}{\|a\|^2}a$.  By the Cauchy-Schwarz inequality, $\|p\|\leq \|z\|\leq r$ and therefore $p$ lies in the closed segment $[-a,a]\cap r\mathbb B$.

If $p\in [-a,0]\cap r\mathbb B $, then $\|p\|< \|p-a\|$ and by the  Pythagorean theorem we conclude that
$$\|z\|^2=\|z-p\|^2+\|p\|^2<\|z-p\|^2+\|p-a\|^2=\|z-a\|^2=d(z, A)^2.$$
This proves that $0$ is an element of $A$ closer to $z$ than $a$, a contradiction. 

On the other hand, if $p\in[0,a]\cap r\mathbb B $, clearly $p\in A\setminus\{a\}$. Using the Pythagorean theorem again,  we obtain that $p$ is an element of $A$ closer to $z$ than $a$. From these contradictions we conclude that  $\|a\|\leq r$ and therefore $a\in A\cap r\mathbb B$. This yields to
$$d(z,A)\leq d(z, A\cap r\mathbb B)\leq \|z-a\|=d(z, A),$$ which proves the equality.






(2)  Let $f:\mathbb R^n\to\mathbb R$ be defined as $f(x)=|d(x,A)-d(x,K)|$, and consider the set
$$U:=f\big(\Int(r\mathbb B)\big)=\{|d(x,A)-d(x,K)|:\|x\|<r\}.$$
Using the continuity of $f$  we obtain that
$$U\subset f(r\mathbb B)=f\Big(\overline{\Int (r\mathbb B)}\Big)\subset \overline {f\big(\Int(r\mathbb B)\big)}=\overline U.$$
Now, using the fact that $f(r\mathbb B)$ is closed, we infer that $\overline U=f(r\mathbb B)$. Thus
$$\sup U=\sup \overline U=\sup f(r\mathbb B),$$
as desired.

(3) Notice that $(1)$ of this proposition and the definition of $d_H$ imply that
\begin{align*}
   d_H\big(A\cap r\mathbb B,K\cap r\mathbb B\big)&=\sup_{x\in\mathbb R^n}|d(x,A\cap r\mathbb B)-d(x,K\cap r\mathbb B)|\\
   &\geq \sup_{x\in r\mathbb B}|d(x,A\cap r\mathbb B)-d(x,K\cap r\mathbb B)|\\
  & =\sup _{x\in r\mathbb B}|d(x,A)-d(x,K)|.
\end{align*}

On the other hand,  if $b\in K\cap r\mathbb B$ then 
\begin{align*}
   d(b, A\cap r\mathbb B)&=|d(b,A\cap r\mathbb B)-d(b,K\cap r\mathbb B)|\\
   &\leq \sup_{x\in r\mathbb B}|d(x,A\cap r\mathbb B)-d(x,K\cap r\mathbb B)|\\
   &=\sup _{x\in r\mathbb B}|d(x,A)-d(x,K)|.
\end{align*}

 Similarly, if $a\in A\cap r\mathbb B$, then  $d(a,K\cap r\mathbb B)\leq \sup _{x\in r\mathbb B}|d(x,A)-d(x,K)|.$ Hence, $d_H(A\cap r\mathbb B, K\cap r\mathbb B)\leq \sup _{x\in r\mathbb B}|d(x,A)-d(x,K)|$. This inequality completes the proof of (3). 

(4) Follows directly from (1),  (2) and (3) of this proposition.

(5) Follows from (2) and (3) of this proposition, and  equivalence (\ref{eq:Lemma-Ananda}).
\end{proof}

 We recall some basic properties of the polar set (\ref{eq:polarsetdefn}) that will be used several times along the paper.
 \begin{itemize}
     \item[(P1)] $A^{\circ}\in\mathcal{K}_0^n$ for every nonempty set $A\subset\mathbb R^n$  and $A^{\circ}\in\mathcal{K}_{(0),b}^n$ provided that $A\in\mathcal{K}_{(0),b}^n$. Moreover, if $A$ and $A^\circ$ belong to $\mathcal{K}_{0,b}^n$, then  $A\in\mathcal{K}_{(0),b}^n$. 
     \item [(P2)] If $A\subset B\subset \mathbb{R}^n$, then $B^\circ\subset A^\circ$.
     \item [(P3)] $\left (\bigcup K_{\alpha}\right)^\circ=\bigcap K_{\alpha}^\circ$.
     \item [(P4)] $\left (\bigcap K_{\alpha}\right)^\circ=\cco\left(\bigcup K_{\alpha}^\circ\right)$.
     \item [(P5)] $A^\circ=A$ if and only if $A=\mathbb B$.
     \item [(P6)] $(\mathbb R^n)^\circ=\{0\}$ and $\{0\}^\circ=\mathbb R^n$.
     \item[(P7)] For every $r>0$, $(rA)^\circ=r^{-1}A^{\circ}$. Particularly, $(r\mathbb B)^{\circ}=\frac{1}{r}\mathbb B$.
 \end{itemize}
 Finally, we recall the \textit{Bipolar Theorem} which is going to play an important role in this work.
 
\begin{itemize}
    \item[(P8)] \textbf{(The Bipolar Theorem)} \textit{For every $A\in\mathcal K^n_0$, $(A^{\circ})^\circ=A$.}
 \end{itemize}

 We refer the reader to \cite[Chapter IV]{Barvinok} and \cite[\S 2.2]{ThompsonAC} for a deeper understanding of the polar map.

In \cite{MilmanRotem2017}, V. Milman and L. Rotem introduced a new operation on $\mathcal{K}_{(0),b}^n.$ It is  called the geometric mean $g(A,K)$ of the compact convex sets $A,K\in\mathcal{K}_{(0),b}^n$. 
Since we will use it in the proof of
our main results, we include its definition  and basic properties.
For $A,K\in\mathcal{K}_{(0),b}^n$ and $m\in\mathbb{N},$ let $(A_m)_m$ and $(H_m)_m$ be sequences on $\mathcal{K}_{(0),b}^n$ defined as 
\begin{equation*}
\begin{matrix}
A_0=A & H_0=K\\
A_{m+1}=\frac{A_m+H_m}{2} & H_{m+1}=\left(\frac{A_m^\circ+H_m^\circ}{2}\right)^\circ.
\end{matrix} 
\end{equation*}
In \cite[Theorem 6]{MilmanRotem2017}, it is proved that $H_m\subseteq A_m,$ for all $m\in\mathbb{N},$ and that both sequences converge to the same limit on $(\mathcal{K}_{(0),b}^n,d_H)$. This allows to define \textit{the geometric mean of $A$ and $K$} as 
\begin{equation}\label{eq:geometric mean}
   g(A,K):=\lim_{m\to\infty}A_m=\lim_{m\to\infty}H_m, 
\end{equation}
where the limit is calculated with respect to $d_H.$
The fundamental properties of $g$ appear in \cite{MilmanRotem2017} and \cite{Rotem2016}. Below, we include those that are relevant for this work. Let $A,K\in\mathcal{K}_{(0),b}^n,$ then $g(A,K)\in\mathcal{K}_{(0),b}^n$ and the following hold:
\begin{itemize}
\item[($\Gamma1$)] $g(A,A)=A$ and $g(A,K)=g(K,A).$   
\item[($\Gamma2$)] $g(A^\circ,K^\circ)=g(A,K)^\circ.$
\item[($\Gamma3$)] $g(A,A^\circ)=\mathbb{B}.$
\item[($\Gamma4$)] $g(A_1,K_1)\subseteq g(A_2,K_2)$  if $A_1\subseteq A_2$ and $K_1\subseteq K_2,$ with $A_i,K_i\in\mathcal{K}_{(0),b}^n$ for $i=1,2.$

\item[($\Gamma5$)] The map $g:(\mathcal{K}_{(0),b}^n,d_H)\times(\mathcal{K}_{(0),b}^n,d_H)\rightarrow (\mathcal{K}_{(0),b}^n,d_H),$ sending $(A,K)$ to $g(A,K)$ is continuous.

\item[($\Gamma6$)] $g(TA,TK)=Tg(A,K)$ for every linear map $T:\mathbb{R}^n\rightarrow\mathbb{R}^n.$

\end{itemize}

Let us notice that in this work, the words \textit{map} and \textit{mapping} do not  mean continuous function. When necessary, we will emphasize if a map is continuous or not.

In Sections \ref{sec:K0n Hilbert cube} and \ref{sec:K0n Z2-AR} we will use several results and notions coming from the theory of retracts,  the theory of $Q$-manifolds and the theory of $G$-spaces. 
Let us quickly recall some of their fundamentals that will be used later. 

\subsection{$\mathrm{ANR}$ and $\mathrm{AR}$ spaces}

We say that a closed subset $Y$ of a topological space $X$ 
is a \textit{retract} of $X$ if there exists a continuous 
map $r:X\rightarrow Y$ such that $r(y)=y$ for all $y\in Y.$
The map $r$ is then called a \textit{retraction}. An \textit{absolute neighborhood retract} ($\mathrm{ANR}$) is a metrizable space $X$ such that for every metrizable space $Z$ containing $X$ as closed subset, there exists a neighborhood $U$ of $X$ in $Z$ and a retraction $r:U\rightarrow X.$ In this case, we write $X\in \mathrm{ANR}.$ If we can always take $U=Z,$ we say that $X$ is an \textit{absolute retract} ($\mathrm{AR}$), and  we write $X\in \mathrm{AR}.$ Clearly, every $\mathrm{AR}$ is a $\mathrm{ANR}$.
A simple example of an $\mathrm{AR}$ space is any closed convex subset of a Banach space (see, e.g., \cite[Theorem 1.5.1]{vanMillbook}).

Let us summarize some basic properties concerning the  $\mathrm{ANR}$ spaces.

\begin{itemize}
    \item [(R1)] Every open subspace of an $\mathrm{ANR}$ is an $\mathrm{ANR}$ (\cite[Chapter 3, Proposition 7.9]{Hu}).
    \item [(R2)] If every point of a separable metric space $X$ has a neighborhood which is an $\mathrm{ANR}$, then $X$ is a $\mathrm{ANR}$ (\cite[Chapter 4, Corollary 10.4]{Borsuk}).
    \item[(R3)] Every $\mathrm{AR}$ is contractible (\cite[Chapter 3, Theorem 7.1]{Hu}). 
    \item [(R4)]   Every contractible $\mathrm{ANR}$ is an $\mathrm{AR}$ (see \cite[Chapter 3, Proposition 7.2]{Hu} or \cite[Corollary 1.6.7]{vanMillbook}).
\end{itemize}

Finally, let us recall that a subset $Y$ of a topological space $X$ is called \textit{homotopy dense} if there exists a homotopy $H:X\times[0,1]\rightarrow X$ such that $H_0$ is the identity on $X$ and $H_t(X)\subseteq Y$ for all $0<t\leq1$.
As usual, $H_t$ denotes the map $H_t:X\to X$ given by $H_t(x):=H(x,t)$ for each $x\in X$. 
The importance of homotopy dense subsets is that they can characterize whether a metrizable space is a $\mathrm{ANR}$, as the following result shows.

\begin{itemize}
    \item [(R5)] Let $X$ be a metrizable space and $Y$ be a homotopy dense subset of $X$. Then $X$ is an $\mathrm{ANR}$ if and only if $Y$ is an $\mathrm{ANR}$ (\cite[Corollary 6.6.7]{Sakai2013book}).
\end{itemize}

For more information on the theory of retracts, we refer the reader to \cite{Borsuk,Hu,vanMillbook,Sakai2013book}.

\subsection{$Q$-Manifolds}

The \textit{Hilbert cube}, denoted by $Q,$ is the topological product $\prod_{i=1}^{\infty}[-1,1]$.  A \textit{Hilbert cube manifold} ($Q$-manifold) is a separable metrizable space that admits an open cover by sets homeomorphic to open subsets of $Q.$ 
In particular, we have the following. 
\begin{fact}\label{fact: open subset of a Q manifold}
Every nonempty open subset of a $Q$-manifold, is a $Q$ manifold too.
\end{fact}

It should be reminded that $Q$ is an $\mathrm{AR}$ (see e.g. \cite[Corollary 1.5.5]{vanMillbook}). This, in combination with properties (R1) and (R2), proves the following fact.

\begin{fact}\label{fact: Q manifolds are ANR}
Every $Q$ manifold is an $\mathrm{ANR}$.
\end{fact}

Let us recall some of the most important  results from the theory of $Q$-manifolds that will be used later in the paper. 

\begin{theorem}\label{thm: X times Q}(\cite[Theorem 44.1]{Chapmanbook})
If $X$ in a $\mathrm{ANR}$ then $X\times Q$
is a $Q$ manifold. 
\end{theorem}

\begin{theorem}
\label{thm:vanMillthm758}(\cite[Theorem 7.5.8]{vanMillbook})
Every compact contractible $Q$-manifold is homeomorphic to $Q.$
\end{theorem}

\begin{theorem}
\label{thm:ChapmanThm151}
(Stability Theorem for $Q$-manifolds \cite[Theorem 15.1]{Chapmanbook}) If $M$ is a $Q$-manifold, then $M\times Q$ is homeomorphic to $M.$  Moreover, the projection map of $M\times Q$ to $M$ is a near homeomorphism.\footnote{A map $f:X\rightarrow Y$ between two metrizable spaces is called \textit{near homeomorphism} if for each open cover $\mathcal{U}=\{U_i\}_{i\in I}$ of $Y$ there exists a homeomorphism $\phi:X\rightarrow Y$ such that for every $x\in X,$ there exists $U_i\in\mathcal{U}$ containing both $f(x)$ and $\phi(x)$\label{fnote:near-homeo}. In particular, if there exists a near homeomorphism between two spaces, these spaces are homeomorphic.} 
\end{theorem}

A continuous surjective  map $f:X\to Y$ between topological spaces is called \textit{proper} if $f^{-1}(K)$ is compact for every compact set $K\subset Y$. 
A proper map $f:X\to Y$ is called \textit{cell-like} (or \textit{CE-map})  if every fiber $f^{-1}(y)$ has trivial shape. Let us recall that a compact metric space $X$ has \textit{trivial shape}  if for every embedding $f:X\to Z$ where $Z$ is a $\mathrm{ANR}$, the image $f(X)$ is contractible in any of its neighborhoods. In particular, if $X$ is contractible, then it has trivial shape. The reader can find more information on this topic in \cite[Subsection 7.1]{vanMillbook}.

The importance of CE-maps is that they often guarantee the existence of a homeomorphism between $Q$-manifolds, as the following Theorem from R. D. Edwards states.

\begin{theorem}
\label{thm:ChapmanThm431}
(\cite[Theorem 43.1]{Chapmanbook}) If $M$ is a $Q$-manifold, $X$ is an $\mathrm{ANR}$ and $f:M\rightarrow X$ is a \textit{CE-map}, then the map  
$$f\times id_{Q}: M\times Q\rightarrow X\times Q$$ 
is a near homeomorphism.\textsuperscript{\ref{fnote:near-homeo}} Here, $id_Q$ denotes the identity map on $Q$.
\end{theorem}

We refer the reader to \cite{Chapmanbook, vanMillbook, Sakai2013book} for a deeper understanding of the theory of $Q$-manifolds. However, in Section~\ref{sec:K0n Hilbert cube} we still require to introduce some other notions from this theory that have not been mentioned here. 

\subsection{$G$-spaces}\label{subsec:G-spaces}

For the proof of Theorem 1,  we will use some important results from the theory of $G$-spaces.  We refer the reader to \cite{Bredon} for a deeper understanding of this topic. However, in order to make our proof as clear as possible, let us recall  some of the basic notions that we will use. 

By a $G$\textit{-space}, we mean a topological space $X$ equipped with a continuous action $\rho:G\times X\rightarrow X$ of a topological group $G$ on $X.$ 
As usual,  $hx$ denotes the image under the action of the pair $(h,x)$, i.e., $hx:=\rho(h,x)$. The set $G(x):=\{hx\in X: h\in G\}$ is called the \textit{orbit} of $x$, and the set of all orbits in $X$, equipped with the quotient topology, is the \textit{orbit space} and it is  denoted by $X/G$. 

A map  $f:X\to Y$ between $G$-spaces is called $G$-\textit{equivariant} (or simply \textit{equivariant}) if it commutes with the action. Namely, if $f(hx)=hf(x)$ for every $x\in X$ and $h\in G$.
 If in addition $f$ is continuous, then we say that $f$ is a $G$-\textit{map}.  
Similarly, if $f$ happens to be a retraction which is also equivariant, then we call it a $G$-\textit{retraction}.

Given a $G$-space $X$, for every $x\in G$  we denote by $G_x$ the \textit{stabilizer} (also called \textit{isotropy group}) of $x$, i.e., 
$$G_x:=\{h\in G : hx=x\}.$$   
An action of a group $G$ on a topological space $X$ is said to be \textit{free}, if $G_x$ is the trivial group for every $x\in X$. This notion leads us to the following definition. 

\begin{definition}\label{def:base free action}
A based-free $G$-space is a $G$-space with a unique point $x_0\in X$, such that $G_{x_0}=G$ and $G_x$ is trivial for every $x\in X\setminus\{x_0\}$. Namely, $x_0$ is a fixed point and the action of $G$ on $X\setminus\{x_0\}$ is free. 
\end{definition}

Every finite group can define a based-free action on the Hilbert cube. Indeed, for a finite group $G$, consider its cone 
$$C(G):=G\times[0,1]/G\times \{0\}.$$ 
Notice that  $C(G)$ is homeomorphic to a contractible compact non trivial polyhedron and therefore it is an $\mathrm{AR}$\footnote{We refer the reader to \cite[Page 117]{vanMillbook} for the definition of a polyhedron and more information about the $\mathrm{AR}$-property on these kinds of spaces.}. 
Thus,  the countable product $\prod_{n\in\mathbb N}C(G)$ is homeomorphic to the Hilbert cube (see, e.g. \cite[Corollary 8.1.2]{vanMillbook}). 
Furthermore, the group $G$ acts continuously on  $C(G)$ by $(h_1, [h_2,t])\to[h_1h_2, t]$ and hence we can define an action of $G$ on the Hilbert cube $\prod_{n\in\mathbb N}C(G)$ by means of the following rule
$$\big(h,([h_n, t_n])_{n\in\mathbb N}\big)\to ([hh_n, t_n])_{n\in\mathbb N}.$$
It is not difficult to prove that this action is based-free and the only fixed point is the point with all its coordinates equal to the vertex $\theta:=[h,0]$. We refer the reader to \cite[Section 3]{Berstein-West} for more details on this construction\footnote{Actually, this construction works for any compact Lie group.}. 
These kinds of actions are called \textit{standard based-free $G$-action on a Hilbert cube}.

\begin{remark}\label{rem: standard z2 action on Q}
Notice that if $G=\mathbb Z_2$, then $C(\mathbb Z_2)$ is homeomorphic with the interval $[-1,1]$ and therefore  the standard  based-free $\mathbb Z_2$-action on $Q$ corresponds  precisely with the one defined by $\big(h, (x_n)_{n\in\mathbb N}\big)\to (hx_n)_{n\in\mathbb N}$, where $h\in \mathbb Z_2:=\{-1,1\}.$
\end{remark}

After this construction, a based-free action $\rho:G\times Q\to Q$ of a finite group on the Hilbert cube is called \textit{standard} if it is conjugate with the standard based-free $G$-action on $Q$. Namely, if there exists an equivariant homeomorphism $\varphi:(Q, \rho)\to (Q, \sigma_G)$, where $\sigma_G$ denotes the  standard based-free $G$-action on $Q$. 

The easiest way to detect if a based-free action on $Q$ is standard is by means of the following theorem of J. West and R. Y. T. Wong.

\begin{theorem}
\label{thm:WestWongThm1}(\cite[Theorem 1]{WestWong79}).
If $G$ is finite, a based-free action  $\rho:G\times Q\rightarrow Q$ is standard if and only if the orbit space  $Q/G$ is an $\mathrm{AR}$.
\end{theorem}

To avoid making this section any longer, all other basic concepts and results  needed in specific spots of the paper will be introduced  as they are needed. 

\section{Homeomorphism type of $\mathcal{K}_0^n$} \label{sec:K0n Hilbert cube}

From now on,  $\mathcal{K}^n$ and all its subspaces are assumed to be metric spaces with respect to the Attouch-Wets metric $d_{AW}$ defined  in equation (\ref{eq:dAW}).

The next lemmas will be used frequently throughout the work.
\begin{lemma}
\label{lem:homotopyK0}
Let $r:(0,1]\rightarrow[0,\infty)$ be a continuous map such that  $r(1)=0$,  $\lim\limits_{t\to 0^+} r(t)=\infty$  and $r(t)>0$ for all
$t\in(0,1)$. Then the map $F:\mathcal{K}_0^n\times[0,1]\rightarrow\mathcal{K}_0^n$ defined by
\begin{align*}
F(K,t):=\begin{cases}
K&\text{if } t=0, \\
K\cap r(t)\mathbb{B}&\text{if }t\in(0,1],
\end{cases}
\end{align*}
is continuous. 
\end{lemma}

\begin{proof}
Let $(A,s)\in \mathcal K^n_0\times [0,1]$ be fixed.
Let us suppose that $s>0$, and let $j_0\geq1$ be an integer such that $r(s)<j_0$. 
Let $\varepsilon>0$ be such that 
$\frac{1}{j+1}<\varepsilon\leq\frac{1}{j}$, for some $j>j_0+1.$  By the continuity of $r$, we can pick $\delta\in(0,s)$ such that
$|r(t)-r(s)|<\frac{\varepsilon}{4},$ if $|t-s|<\delta.$ 
Now, let $(K,t)\in \mathcal K^n_0\times[0,1]$ be any pair satisfying $d_{AW}(K,A)<\frac{\varepsilon}{4}$ and $|t-s|<\delta$.
  Observe that, by our choice of $\delta$, $r(t)<j$ and $t>0$. Therefore,
\begin{align}
\label{eq:distAj}&F_s(A)=A\cap r(s)\mathbb{B}=A\cap r(s)\mathbb{B}\cap j\mathbb{B}\subseteq A\cap j\mathbb{B}\\
\label{eq:distKj}&F_t(K)=K\cap r(t)\mathbb{B}=K\cap r(t)\mathbb{B}\cap j\mathbb{B}\subseteq K\cap j\mathbb{B}.
\end{align}
Recall that $F_s(A)$ and $F_t(K)$ stand for $F(A,s)$ and $F(K,t)$, respectively. 
Also, since $d_{AW}(K,A)<\frac{\varepsilon}{4}<\frac{1}{j}$, we can use Lemma~\ref{lem:car-dAW}-(4) to infer that
\begin{equation}
\label{eq:distaux1}
d_H(K\cap j\mathbb{B},A\cap j\mathbb{B})<\frac{\varepsilon}{4}.
\end{equation} 
Now, consider $b\in A\cap r(s)\mathbb{B}.$ If $b\in r(t)\mathbb{B},$ then $d\big(b,A\cap r(t)\mathbb{B}\big)<\frac{\varepsilon}{4}.$ Otherwise, $\|b\|>r(t)$ and 
$\frac{r(t)}{\|b\|}b \in A\cap r(t)\mathbb B$. Hence,
$$
d\big(b,A\cap r(t)\mathbb{B}\big)\leq\left\Vert b-\frac{r(t)}{\|b\|}b\right\Vert\leq r(s)-r(t)<\frac{\varepsilon}{4}.
$$ 
Similarly, $d(b',A\cap r(s)\mathbb{B})<\frac{\varepsilon}{4}$ for all $b'\in A\cap r(t)\mathbb{B}.$ Thus,
$d_H\big(F_s(A), F_t(A)\big)\leq\frac{\varepsilon}{4}$
which, by the triangle inequality, yields to
\begin{align}
\label{eq:triangle-ineq}
d_H\big(F_s(A), F_t(K)\big)&
\leq \frac{\varepsilon}{4}+d_H\big(F_t(A), F_t(K)\big).
\end{align}
We claim that $d_H\big(F_t(A), F_t(K)\big)\leq\frac{\varepsilon}{2}.$ Indeed, let $a\in A\cap r(t)\mathbb{B}\subset A\cap j\mathbb B$. By (\ref{eq:distaux1}), there is $x\in K\cap j\mathbb{B}$ such that $\|x-a\|<\frac{\varepsilon}{4}.$  If $x\in r(t)\mathbb{B},$ then $d(a,K\cap r(t)\mathbb{B})<\frac{\varepsilon}{4}.$ Otherwise, $\|x\|>r(t)$ and 
$\frac{r(t)}{\|x\|}x\in K$.
In this case,
$0<\|x\|-r(t)<\frac{\varepsilon}{4}$ and
\begin{equation*}
\left\Vert a-\frac{r(t)}{\|x\|}x\right\Vert\leq \left\Vert a-\frac{r(t)}{\|x\|}a\right\Vert+\left\Vert \frac{r(t)}{\|x\|}a-\frac{r(t)}{\|x\|}x\right\Vert
\leq\frac{\varepsilon}{4}\left(\frac{\|a\|}{\|x\|}+\frac{r(t)}{\|x\|}\right)
<\frac{\varepsilon}{2}.
\end{equation*}

Thus $d(a,K\cap r(t)\mathbb{B})<\frac{\varepsilon}{2}.$ In a similar way, we can prove that $d(a',A\cap r(t)\mathbb{B})<\frac{\varepsilon}{2}$ for all $a'\in K\cap r(t)\mathbb{B}.$ Therefore  $d_H(F_t(A), F_t(K))\leq\frac{\varepsilon}{2}.$ This, in combination with  (\ref{eq:distAj}), (\ref{eq:distKj}) and  (\ref{eq:triangle-ineq}), implies that 
$
d_H(F_s(A)\cap j\mathbb{B},F_t(K)\cap j\mathbb{B})<\varepsilon.
$
Hence, by  Lemma \ref{lem:car-dAW}-(5), 
$d_{AW}\big(F_s(A), F_t(K)\big)<\varepsilon,$ 
and therefore $F$ is continuous at $(A,s)$ whenever $s\neq 0.$

Now, let us consider the case when $s=0$. Let $\varepsilon>0$ and assume  that $\varepsilon\in\left(\frac{1}{j+1},\frac{1}{j}\right]$ for some integer $j\geq1.$ By our hypothesis, there is 
$\delta\in (0,\varepsilon)$ such that $r(t)>j$ for all 
$t\in(0,\delta)$. Let $(K,t)\in \mathcal K^n_0\times[0,\delta)$  be such that $d_{AW}(A,K)<\varepsilon$. Clearly, if $t=0$, $d_{AW}\big(F(A,0), F(K,t)\big)=d_{AW}(A, K)<\varepsilon$.
On the other hand, if $t>0$, then
$K\cap j\mathbb{B}=F(K,t)\cap j\mathbb{B}$. 
This, in combination with  Lemma~\ref{lem:car-dAW}-(5) 
and the fact that $d_{AW}(A, K)<\varepsilon$, 
implies that
\begin{align*}
d_H\big(F(A, 0)\cap j\mathbb{B},F(K, t)\cap j\mathbb{B}\big)=  d_H(A\cap j\mathbb{B},K\cap  j\mathbb{B})<\varepsilon.  
\end{align*}
Using Lemma~\ref{lem:car-dAW}-(5) again, we conclude  that $d_{AW}\big(F(A,0),F(K,t)\big)<\varepsilon,$ which proves the continuity of $F$  at $(A,0)$.
\end{proof}

\begin{lemma}\label{lem:homotopy H}
Let $H:\mathcal{K}_0^n\times[0,1]\rightarrow\mathcal{K}_0^n$ be defined as $H(K,t)=K+t\mathbb{B},$ for $K\in\mathcal{K}_0^n$ and $t\in[0,1].$
Then $H$ is continuous.
\end{lemma}

\begin{proof} 

By Fact~\ref{fact:topologies are the same}, it is enough to prove that $H$ is continuous with respect to the Fell topology on $\mathcal K_0^n$. Let $U\subset \mathbb R^n$ be an open set and consider $(K,t)\in\mathcal{K}_0^n\times [0,1]$ with $H(K,t)\in U^-$. Thus we can find $a\in K$ and $b\in \mathbb B$ such that $a+tb\in U$. By the continuity of the scalar product and the sum operation on $\mathbb R^n$,  we can find open neighborhoods $V$ and $W$ of $a$ and $t$, respectively, such that
$$a'+t'b\in U\quad\text{ for every }a'\in V\text{ and }t'\in W.$$
This proves that $H(K', t')\in U^-$ for every $K'\in V^-$ and $t'\in W$. 

On the other hand, if $M\subset \mathbb R^n$ is compact and $K+t\mathbb B=H(K,t)\in (\mathbb R^n\setminus M)^+$, we can  find $\delta>0$ small enough such that
$$(K+t\mathbb B+\delta \mathbb B)\cap (M+\delta\mathbb B)=\emptyset.$$
Clearly, $K\cap (M+(t+\delta)\mathbb B)=\emptyset$. Thus, since $C:=(M+(t+\delta)\mathbb B)$ is compact,  $(\mathbb R^n\setminus C)^+$ is an open neighborhood of $K$.
Finally, let $A\in (\mathbb R^n\setminus C)^+$ and $s\in[0,1]$ with $|s-t|<\delta$. If $a\in A$, $b\in \mathbb B$ and $x\in M$ are arbitrary, then   the following inequality holds.
$$\|a+sb-x\|\geq\|a-x\|-\|sb\|>t+\delta-(t+\delta)>0$$
Hence, $H(A,s)\in (\mathbb R^n\setminus M)^+$, which proves the continuity of  $H$. 
\end{proof}

Before proving the main results of this section, let us recall some facts concerning the topology of certain subspaces of $\mathcal K^n$. 

\begin{fact}
\label{fc:proposition35SakYag}
(\cite[Proposition 3.5]{SakaiYaguchi2006})
The subspace $\mathcal{K}_b^n\subset\mathcal{K}^n$ is open.
\end{fact}

Let $cb(n):=\{A\in\mathcal{K}_{b}^n:\Int(A)\neq\emptyset\}$. The  following results from \cite{AntonyanNatalia} will be used in the proof of the proposition below.

\begin{fact}\label{fc:cbn-topology}(\cite[Corollary 3.11]{AntonyanNatalia}) The space $\big(cb(n),d_H\big)$ is homeomorphic to $Q\times\mathbb{R}^{\frac{n(n+3)}{2}}.$
\end{fact}

\begin{fact}
(\cite[Lemma 3.1]{AntonyanNatalia})
\label{lem:cbn-open}
Let $A,C\in cb(n)$ and let $x_0\in A$ be such that $x_0+2\varepsilon\mathbb{B}\subseteq A,$ for certain $\varepsilon>0.$ If $d_H(A,C)<\varepsilon,$ then $x_0+\varepsilon\mathbb{B}\subseteq C$.
\end{fact}

As a direct consequence of Fact~\ref{lem:cbn-open}, we obtain the following. 

\begin{fact}\label{fact: K(0) is open}
$(\mathcal{K}_{(0),b}^n, d_H)$ is an open set in 
$(cb(n),d_H)$.
\end{fact}

\begin{proposition} \label{prop:K0AR}
For any integer $n\geq2$, the following statements hold:
\begin{enumerate}
\item $\mathcal{K}_{(0)}^n$ and $\mathcal{K}_{(0),b}^n$  are open subsets of $\mathcal{K}^n.$ Furthermore, 
$\mathcal{K}_{(0),b}^n$ is a $Q$-manifold.
\item $\mathcal{K}_{0}^n$ is a contractible space. 
\item $\mathcal{K}_{0}^n$  is  a compact AR
\end{enumerate}
\end{proposition}
\begin{proof}
(1) First, we show that $\mathcal{K}_{(0)}^n\subset\mathcal{K}^n$ is open. Let $A\in\mathcal{K}_{(0)}^n$ and pick $\varepsilon_0>0$ such that $\varepsilon_0\mathbb{B}\subset A.$ We shall prove that every  $K\in\mathcal{K}^n$ with $d_{AW}(A,K)<\varepsilon_0$ belongs to $\mathcal{K}_{(0)}^n$. Assume, without lost of generality, that $\frac{1}{j+1}<\varepsilon_0\leq\frac{1}{j}$ for some integer $j\geq1.$  By (\ref{eq:Lemma-Ananda}),
\begin{equation*}
    \sup_{\|x\|<j}|d(x,A)-d(x,K)|<\varepsilon_0.
\end{equation*}
Let $p\in K$ be the closest point to $0$. If $p\neq 0$, then $w:=\frac{-\varepsilon_0}{\|p\|}p\in\varepsilon_0\mathbb B\subset A$ and $d(w, K)=\varepsilon_0+\|p\|>\varepsilon_0$.
Therefore
$$\varepsilon_0<d(w, K)\leq \sup_{\|x\|<j}|d(x,A)-d(x,K)|<\varepsilon_0,$$
a contradiction. This implies that $p=0$ and hence $K\in\mathcal K^n_0$. We can now use
Lemma~\ref{lem:car-dAW} to infer that
 $$d_{H}(A\cap j\mathbb{B},K\cap j\mathbb{B})=\sup_{\|x\|<j}|d(x,A)-d(x,K)|<\varepsilon_0.$$ 
Observe that $A\cap j\mathbb{B}\in\mathcal{K}_{(0),b}^n$ and  $\varepsilon_0\mathbb{B}\subset A\cap j\mathbb{B}.$  Thus, from  Fact \ref{lem:cbn-open}, it follows that $\frac{\varepsilon_0}{2}\mathbb{B}\subset K\cap j\mathbb{B}\subset K$ and therefore $K\in\mathcal{K}_{(0)}^n$, as desired.

In order to prove that $\mathcal{K}_{(0),b}^n$ is  open, let us recall that $\mathcal{K}_{b}^n\subset\mathcal{K}^n$ is an open set (Fact \ref{fc:proposition35SakYag}). Thus $\mathcal{K}_{(0),b}^n=\mathcal{K}_{b}^n\cap\mathcal{K}_{(0)}^n$  is an intersection of two open sets and therefore it is also open in $\mathcal K^n_0$.

Recall that the metrics $d_H$ and $d_{AW}$ generate the same topology on $\mathcal{K}_{(0),b}^n$ (Fact \ref{fact:d_H y d_AW coincinden en K_b}). In consequence, to prove that 
$\mathcal{K}_{(0),b}^n$ is a $Q$-manifold, it is enough to show that 
$(\mathcal{K}_{(0),b}^n,d_{H})$ is a $Q$-manifold.  Indeed, 
from Fact \ref{fact: K(0) is open},
 it follows that  $(\mathcal{K}_{(0),b}^n,d_{H})$ is an open set in $(cb(n),d_H).$  Since $(cb(n),d_H)$ is a $Q$-manifold (Fact \ref{fc:cbn-topology} and Theorem~\ref{thm: X times Q}), $(\mathcal{K}_{(0),b}^n,d_{H})$  must be a $Q$-manifold too (Fact~\ref{fact: open subset of a Q manifold}).

(2) Notice that if we let $r(t)=\frac{1-t}{t},$ $t\in(0,1],$ be the map of Lemma \ref{lem:homotopyK0}, then the corresponding map $F:\mathcal{K}_0^n\times[0,1]\rightarrow\mathcal{K}_0^n$ is a  homotopy such that $F(A,0)=A$ and $F(A,1)=\{0\}$ for all $A\in\mathcal{K}_0^n.$ Hence, $\mathcal{K}_0^n$ is contractible to $\{0\}.$ 

(3) We will show that $\mathcal{K}_{0}^n$ is compact with respect to the Fell topology $\tau_F,$ then the compactness of $(\mathcal{K}_{0}^n,d_{AW})$  will follow from the fact that the topologies $\tau_F$ and $\tau_{AW}$ coincide on $\mathcal{K}_{0}^n$. Since $\mathcal{K}^n_*=\mathcal{K}^n\cup\{\emptyset\}$ endowed with $\tau_F$ is the Alexsandroff one-point compactification of $(\mathcal{K}^n,\tau_F)$ (\cite[Proposition 1]{SakaiYang2007}), it is enough to prove that $\mathcal{K}_{0}^n$ is a closed subset of $\mathcal{K}^n_*.$  But this follows directly from the fact that the complement $\mathcal{K}^n_*\setminus\mathcal{K}_{0}^n=(\mathbb{R}^n\setminus\{0\})^+$ is an open set in $\mathcal{K}^n_*.$ 

We turn to prove that $\mathcal{K}_{0}^n$ is an $\mathrm{AR}$. To do so, observe that by (1) of this proposition, $\mathcal{K}_{(0),b}^n$ is a $Q$-manifold, so it is an $\mathrm{ANR}$ (Fact~\ref{fact: Q manifolds are ANR}). Even more, by the lemmas \ref{lem:homotopyK0} and \ref{lem:homotopy H}, the map $h:\mathcal{K}_0^n\times[0,1]\rightarrow\mathcal{K}_0^n,$ defined by $h(A,t)=A\cap\frac{1-t}{t}\mathbb{B}+t\mathbb{B},$  $t\neq 0,$ and $h(A,0)=A,$ is a homotopy such that $h_0$ is the identity map of $\mathcal{K}_{0}^n,$ and $h(\mathcal{K}_{0}^n,t)\subseteq\mathcal{K}_{(0),b}^n$ for all $t\in(0,1].$ Thus,  $\mathcal{K}_{(0),b}^n$ is a homotopy dense subset of $\mathcal{K}_{0}^n,$ and therefore, by Property (R5), $\mathcal{K}_{0}^n$ is an $\mathrm{ANR}$. Since $\mathcal{K}_{0}^n$ is a contractible space (see (2) of this proposition), it must be an $\mathrm{AR}$ (Property (R4)).
\end{proof}

In order to prove that $\mathcal K^n_0$ is homemorphic to $Q$, let us recall what a $Z$-set is. A closed subset $Y$ of a metric space $(X,d)$ is called a $Z$\textit{-set} if the set $\{f\in C(Q,X): f(Q)\cap Y=\emptyset\}$ is dense in $C(Q,X).$ Here, $C(Q,X)$ denotes the space of continuous maps from $Q$ to $X$ endowed with the compact-open topology. In particular, if for every $\varepsilon>0,$ there exists a continuous map $f:X\rightarrow X\setminus Y$ 
such that $d(x,f(x))<\varepsilon$, then $Y$ is a $Z$-set. 

The following fact on $Q$-manifolds will be used in the proof of Theorem \ref{thm:K0cuboHilbert}. It shows that any locally compact ANR containing a ``big enough"  $Q$-manifold as an open set, must be a $Q$-manifold.
\begin{fact}
\label{fc:Torun80}
Let $X$ be a locally compact $\mathrm{ANR}$ containing a $Z$-set $Y$. If $X\setminus Y$ is a $Q$-manifold, then $X$ must be a $Q$-manifold too. 
\end{fact}
The proof of this fact appears in the first paragraph of \cite[\textsection 3]{Torunczyk1980}.

\begin{theorem}
\label{thm:K0cuboHilbert}
$\mathcal{K}_{0}^n,$ $n\geq2,$ is homeomorphic to $Q$.
\end{theorem}
\begin{proof}
We already proved that $\mathcal{K}_{0}^n$ is a compact contractible AR (Proposition \ref{prop:K0AR}). Thus, since  $Q$ is the only $Q$-manifold which is both compact and contractible (Theorem \ref{thm:vanMillthm758}), it is enough to show that $\mathcal{K}_{0}^n$ is a $Q$-manifold. To this end, observe that $\mathcal{K}_{(0),b}^n$ is an open set in $\mathcal{K}_0^n$ (Proposition \ref{prop:K0AR}-(1)). Therefore, $\mathcal{K}_0^n\setminus\mathcal{K}_{(0),b}^n$ is closed in $\mathcal{K}_0^n.$ Even more $\mathcal{K}_0^n\setminus\mathcal{K}_{(0),b}^n$ is a $Z$-set. To prove this, observe that the map $h,$ used in the proof of Proposition \ref{prop:K0AR}-(3), is a homotopy such that $h_0$ is the identity map of $\mathcal{K}_{0}^n,$ and $h_t(\mathcal{K}_{0}^n)\subseteq\mathcal{K}_{(0),b}^n$ for all $t\in(0,1].$ Thus, for every $\varepsilon>0,$ we can pick $\delta\in(0,1)$ such that  $d_{AW}\big(A,h_t(A)\big)<\varepsilon$ for all $t\in(0,\delta).$ In consequence, the map $h_t:\mathcal{K}_{0}^n\rightarrow\mathcal{K}_{(0),b}^n,$ $t\in(0,\delta),$ is continuous and $d_{AW}\big(A,h_t(A)\big)<\varepsilon$ for all $A\in\mathcal{K}_{0}^n$. Hence, $\mathcal{K}_0^n\setminus\mathcal{K}_{(0),b}^n$ is a $Z$-set. This, in combination with the fact that $\mathcal{K}_{(0),b}^n\subset\mathcal{K}_0^n$ is a $Q$-manifold (Proposition \ref{prop:K0AR}-(1)), proves that $\mathcal{K}_0^n$ is a $Q$-manifold too (Fact \ref{fc:Torun80}).
\end{proof}

It is worth noting that, as a consequence of Proposition \ref{prop:K0AR} and Theorem \ref{thm:K0cuboHilbert}, $\mathcal{K}_{(0)}^n\subset\mathcal{K}_{0}^n$ is an open set in a Hilbert cube, and therefore it is a $Q$-manifold (Fact \ref{fact: open subset of a Q manifold}). Theorem \ref{thm:K0cuboHilbert} also allows us to determine the homeomorphism type of $(\mathcal{K}_{0,b}^n,d_{AW}).$ In order to do that,  it should be reminded that $Q$ is a homogeneous space (see e.g. \cite[Theorem 6.1.6]{vanMillbook}) and therefore $Q\setminus\{x\}$ and $Q\setminus\{y\}$ are homeomorphic for any $x,y\in Q$. In what follows,  $Q\setminus\{*\}$ denotes the Hilbert cube with a point removed.

\begin{corollary}
\label{cor:topo-strucK0b}
$\mathcal{K}_{0,b}^n$ is homeomorphic to $Q\setminus\{*\}.$
\end{corollary}
\begin{proof}
It is a very well-known result that $Q\setminus\{*\}$ and $Q\times [0,\infty)$ are homeomorphic spaces (see e.g. the proof of \cite[Theorem 12.2]{Chapmanbook}). Thus, in order to achieve our goal, it is enough to prove that $(\mathcal{K}_{0,b}^n, d_{AW})$ is homeomorphic to $Q\times [0,\infty)$.
Recall that $\mathcal{K}_{b}^n$  is an open set in $(\mathcal{K}^n,d_{AW})$ (Fact \ref{fc:proposition35SakYag}) and 
 $\mathcal{K}_{0,b}^n=\mathcal{K}_0^n\cap\mathcal{K}_{b}^n.$
 Thus, $\mathcal{K}_{0,b}^n$ is an open set in $\mathcal{K}_0^n$. Hence, by Theorem \ref{thm:K0cuboHilbert}, 
$\mathcal{K}_{0,b}^n$ is a $Q$-manifold.
Moreover, since the metrics $d_{AW}$ and $d_H$ generate the same topology on $\mathcal{K}_{b}^n$ (Fact \ref{fact:d_H y d_AW coincinden en K_b}), both $(\mathcal{K}_{0,b}^n,d_{H})$ and $(\mathcal{K}_{0,b}^n,d_{AW})$ are homeomorphic $Q$-manifolds. Then, to prove the result, it is enough to show that $(\mathcal{K}_{0,b}^n,d_{H})$ is homeomorphic to $Q\times [0,\infty)$. 
Consider the map  
$\nu:(\mathcal{K}_{0,b}^n,d_{H})\rightarrow[0,\infty)$ defined as $\nu(A)=\max_{a\in A}\|a\|,$ for $A\in\mathcal{K}_b^n.$ In \cite[Lemma 4.2]{AntonyanNatalia}, it was shown that $\nu$ is a continuous map. Since $\nu(r\mathbb{B})=r$ for every $r\in[0,\infty),$ we also know that $\nu$ is surjective. Furthermore, if $\mathcal{W}\subset[0,\infty)$ is a compact set, then $\nu^{-1}(\mathcal{W})\subset\mathcal{K}_{0,b}^n$ is a closed bounded set. So, by the Blaschke Selection Theorem (\cite[Theorem 1.8.6]{Schneider1993}), $\nu^{-1}(\mathcal{W})$ must be compact. In addition, $\nu^{-1}(r)$ is contractible for all $r\in[0,\infty).$ Indeed, let $f:\nu^{-1}(r)\times[0,1]\rightarrow\nu^{-1}(r)$ be  defined as $f(A,t)=(1-t)A+tr\mathbb{B},$ for $t\in[0,1]$ and $A\in\nu^{-1}(r).$ If $A\in\nu^{-1}(r)$ and $a_0\in A$ is such that $\nu(A)=\|a_0\|=r,$ then 
$$
\left\Vert(1-t)a_0+\frac{rt}{\|a_0\|}a_0\right\Vert=r.
$$
Hence, $f(A,t)\in\nu^{-1}(r)$ for all $t\in[0,1],$ which proves that $f$ is well-defined. The  continuity of $f$ follows directly from \cite[Theorem 2.7.5]{WebsterR} (see also \cite[\textsection 1.8]{Schneider1993}). Clearly, $f_0$ is the identity map of $\nu^{-1}(r)$ and $f_1(A)=r\mathbb{B}$ for all $A\in\nu^{-1}(r).$ Therefore, $f$ is a homotopy and $\nu^{-1}(r)$ is a contractible space. We have thus shown that $\nu$ is a CE-map. Hence, by Theorem \ref{thm:ChapmanThm431}, $Q\times(\mathcal{K}_{0,b}^n,d_{H})$ is homeomorphic with $Q\times[0,\infty).$
However, since $(\mathcal{K}_{0,b}^n,d_{H})$ and $Q\times(\mathcal{K}_{0,b}^n,d_{H})$ are homemorphic too
(Theorem \ref{thm:ChapmanThm151}), then $(\mathcal{K}_{0,b}^n,d_{H})$ is homeomorphic with
$Q\times[0,\infty)$, as desired.  
\end{proof}

\section{Topology of the polar involution on $\mathcal K^n_0$}
\label{sec:K0n Z2-AR}

 In Section~\ref{sec:prelim} we have recalled some fundamental properties of the polar set $A^\circ$.
Notice that properties $(P1)$ and $(P8)$ guarantee that  the polar mapping
\begin{align*}
\alpha:\mathcal{K}_0^n&\rightarrow\mathcal{K}_0^n\\
\alpha(A)&= A^\circ
\end{align*}
is a well-defined bijective map with the property that  $\alpha(\alpha(A))=A,$ for all $A\in\mathcal{K}_0^n.$ 
Moreover, by property (P5), $\alpha$ has a unique fixed point: the Euclidean ball $\mathbb{B}$. 
Furthermore, in \cite[Theorem 7.2]{Wijsman1966}, R. Wijsman shows the following:
\begin{theorem}
\label{thm:Wijsman}
The polar mapping $\alpha:\mathcal{K}_0^n\rightarrow\mathcal{K}_0^n$
is continuous with respect to the Wijsman topology $\tau_{W}$ on $\mathcal{K}_0^n.$
\end{theorem}
This theorem, in combination with Fact~\ref{fact:topologies are the same} (cf. \cite[Theorem 3.1.4]{Beer1993}),  yields the following.
\begin{fact}\label{fact: continuidad polar daw }
The polar map $\alpha$ is continuous with respect to $d_{AW}.$ 
\end{fact}

 These observations are summarized in the next proposition.
\begin{proposition}
\label{prop:polar-continuous}
The polar map $\alpha:(\mathcal{K}_0^n, d_{AW})\rightarrow(\mathcal{K}_0^n, d_{AW})$ is a  based-free involution. Moreover, $\mathbb B\in\mathcal K^n_0$ is the only fixed point of $\alpha$.
\end{proposition}

After Theorem~\ref{thm:K0cuboHilbert}, we know that $\mathcal{K}_0^n$ is homeomorphic to the Hilbert cube $Q$. This fact, in combination with Proposition \ref{prop:polar-continuous}, shows that the polar mapping determines a based-free involution on a Hilbert cube. Bearing in mind Anderson's problem on the characterization of all based-free involutions on $Q$, it is then natural to ask for the relation between the involution $\sigma:Q\rightarrow Q,$ given by $\sigma(x)=-x,$ $x\in Q,$ and $\alpha:\mathcal{K}_0^n\rightarrow\mathcal{K}_0^n$. This section is dedicated to prove Theorem~\ref{Theo:main 1}, namely, that $\alpha$ and $\sigma$ are conjugate.

To achieve this, we will rely on the theory of $G$-spaces that we introduced in Subsection~\ref{subsec:G-spaces}.  We start by noticing that every involution $\beta:X\to X$ on a topological space $X$, induces a continuous action of the  group $\mathbb Z_2=\{-1,1\}$ as follows
$$(1,x)\to x\;\;\text{ and }\;\;(-1,x)\to \beta(x).$$
Observe that the $\mathbb Z_2$-action is based-free if and only if the involution $\beta$ is based-free. 
In this situation, the orbit space $X/\mathbb Z_2$ is precisely the quotient space $X/\beta$ induced by the involution. 
In particular, the involution $\alpha$ turns $\mathcal K^n_0$ into a based-free $\mathbb Z_2$-space (see Subsection \ref{subsec:G-spaces}). 

In order to prove Theorem~\ref{Theo:main 1}, by Theorem~\ref{thm:WestWongThm1}, it is enough to show that the quotient space
$\mathcal K^n_0/\alpha$ induced by the polar mapping is an $\mathrm{AR}$.
However, by \cite[Theorem 8]{Antonyan1990},
$\mathcal K^n_0/\alpha$ is an $\mathrm{AR}$ provided that  $\mathcal K^n_0$ is a $\mathbb Z_2$-$\mathrm{AR}$. Let us recall what this means.

 A set $S\subset X$ is called \textit{invariant}, if $hx\in S$ for every $(h,x)\in G\times S$.
We say that a metrizable $G$-space X is a \textit{$G$-equivariant absolute neighborhood retract} (denoted by $G$-$\mathrm{ANR}$) if for any metrizable $G$-space $Z$ containing $X$ as an invariant closed subset, there
exists an invariant neighborhood $U$ of $X$ in $Z$ and a $G$-retraction $r : U\rightarrow X$. If we can always take $U = Z$, then we say $X$ is a \textit{$G$-equivariant absolute retract} (denoted by $G$-$\mathrm{AR}$).

To prove that $\mathcal K^n_0$ is a $\mathbb Z_2$-$\mathrm{AR}$, we will use the following result of S. Antonyan (see \cite[Corollary 3.6]{Antonyan2005}).

\begin{theorem}\label{the: Antonyan G contactible} Let $X\in \mathrm{AR}$ be a metrizable based-free $G$-space, where $G$ is a compact Lie group. If $X$ is $G$-contractible, then $X$
is a $G$-$\mathrm{AR}$.
\end{theorem}

In the previous theorem,  a $G$-space $X$ is called \textit{$G$-contractible} if there exists a $G$-fixed point $x_0\in X$ and a homotopy $\varphi:X\times[0,1]\rightarrow X$ such that $\varphi_0$ is the identity on $X$,  $\varphi_t$ is a $G$-map for every $t\in[0,1]$ and $\varphi_1(x)=x_0$ for every $x\in X$. 

Notice that $\mathcal K^n_0$ is indeed a based-free $\mathbb Z_2$-space, and $\mathbb B$ is the unique fixed point.  Furthermore,  $\mathcal K^n_0\in \mathrm{AR}$ (Proposition~\ref{prop:K0AR}). Thus,  if we translate Theorem~\ref{the: Antonyan G contactible} to our situation, we obtain the following remark that summarizes all the previous paragraphs. 

\begin{remark}\label{rem: prueba Main 1}
$\mathcal{K}^n_0$ is $\mathbb Z_2$-contractible if
 there exist a homotopy $\varphi:\mathcal{K}^n_0\times[0,1]\to\mathcal K^{n}_0$ such that $\varphi_0(A)=A$,  $\varphi_1(A)=\mathbb B^n$,  and $\varphi_t(A^\circ)=\varphi_t(A)^\circ$ for every $A\in\mathcal K^n_0$ and every $t\in [0,1]$.  Furthermore, if $\mathcal{K}^n_0$ is $\mathbb Z_2$-contractible, then the following statements hold.
\begin{enumerate}
    \item $\mathcal{K}^n_0$ is a $\mathbb Z_2$-$\mathrm{AR}$.
    \item $\mathcal{K}^n_0/\alpha$ is an $\mathrm{AR}$.
    \item Theorem~\ref{Theo:main 1} is true. 
\end{enumerate}
\end{remark}

In the following pages, we will prove that $\mathcal K^n_0$ is indeed $\mathbb Z_2$-contractible. 

Throughout the rest of this section we will make extensive use of the map  $\psi:\mathcal{K}_0^n\times[0,1]\rightarrow\mathcal{K}_0^n$ defined as
\begin{equation}\label{eq:funcion psi}
    \psi(A,t):=\begin{cases}
A\quad\text{if } t=0, \\
(A+t\mathbb{B})\cap\frac{1-t}{t}\mathbb{B}\quad\text{if }t\in(0,1].
\end{cases}
\end{equation}

\begin{lemma}
\label{lem:psimap}
The map $\psi:\mathcal{K}_0^n\times[0,1]\rightarrow\mathcal{K}_0^n$ defined in (\ref{eq:funcion psi})
is continuous. Moreover, for every $t\in \left[\frac{\sqrt5-1}{2}, 1\right)$ and every $A\in\mathcal{K}_0^n,$  $\psi(A,t)=\frac{1-t}{t}\mathbb{B}.$ 
\end{lemma}
\begin{proof}
First, notice that the continuity of the map $H(A,t)=A+t\mathbb{B},$ $A\in\mathcal{K}_0^n$ and $t\in[0,1],$ was proved in Lemma~\ref{lem:homotopy H}.
Also, observe that if we let $r(t)=\frac{1-t}{t}$ be the map of Lemma \ref{lem:homotopyK0}, then the associated map $F$ is such that $\psi(A,t)=F(H(A,t),t)$ for every $A\in\mathcal{K}_0^n$ and $t\in[0,1].$ Hence, the continuity of $\psi$ follows from the continuity of $H$ and $F.$
To show that $\psi(A,t)=\frac{1-t}{t}\mathbb{B}$ for $\frac{\sqrt5-1}{2}\leq t<1,$ notice that in this interval we always have that $t\geq\frac{1-t}{t}$. Hence,  for every $A\in\mathcal{K}_0^n$ and $t\in\big[\frac{\sqrt5-1}{2}, 1\big)$, we conclude that $\frac{1-t}{t}\mathbb{B}=\{0\}+\frac{1-t}{t}\mathbb{B}\subseteq A+t\mathbb{B}$ and therefore $\frac{1-t}{t}\mathbb{B}\subseteq\psi(A,t)\subseteq\frac{1-t}{t}\mathbb{B}.$ 
\end{proof}

\begin{lemma}
\label{lem:kappa-map}
For each $t\in(0,1)$ and $A\in\mathcal{K}_0^n,$  let $\kappa_t(A)$ be the set 
\begin{equation}\label{eq:kappa}
    \kappa_t(A):=\cco\left(\bigcup_{a\in A}\left[0,\frac{1}{1+t\|a\|}a\right]\cup t\mathbb{B}\right).
\end{equation}

Then  $\kappa_t(A)\subseteq\psi_t(A^\circ)^\circ.$   Recall that $\psi_t$ stands for the map  $\psi(\cdot, t)$ and $\psi$ is the map defined in (\ref{eq:funcion psi}).
\end{lemma}
\begin{proof} Let   $A\in\mathcal{K}_0^n$ and observe that 
 $A=\bigcup_{a\in A}[0,a]$. Then,  by property (P3), 
$A^\circ=\bigcap_{a\in A}\left[0,a\right]^\circ$ and therefore
$$A^\circ+t\mathbb{B}=\bigcap_{a\in A}[0,a]^\circ+t\mathbb{B}\subseteq\bigcap_{a\in A}\left\{[0,a]^\circ+t\mathbb{B}\right\}.$$ Now we can use some of the basic properties of the polar set to infer that
\begin{align*}
\psi_t(A^\circ)^\circ&\supseteq\left(\bigcap_{a\in A}\left\{[0,a]^\circ+t\mathbb{B}\right\}\cap \frac{1-t}{t}\mathbb{B}\right)^\circ\\
&=\cco\left(\left(\bigcap_{a\in A}\left\{[0,a]^\circ+t\mathbb{B}\right\}\right)^\circ\cup\frac{t}{1-t}\mathbb{B}\right)\\
&=\cco\left(\bigcup_{a\in A}\left\{[0,a]^\circ+t\mathbb{B}\right\}^\circ\cup\frac{t}{1-t}\mathbb{B}\right)\\
&=\cco\left(\bigcup_{a\in A}\left[0,\frac{1}{1+t\|a\|}a\right]\cup\frac{t}{1-t}\mathbb{B}\right)\\
&\supseteq\cco\left(\bigcup_{a\in A}\left[0,\frac{1}{1+t\|a\|}a\right]\cup t\mathbb{B}\right).\\
\end{align*}
Therefore $\kappa_t(A)\subseteq\psi_t(A^\circ)^\circ$ for all $A\in\mathcal{K}_0^n.$
\end{proof}
\begin{theorem}
\label{thm:K0-Z2-contractible}
 The space  $\mathcal{K}_0^n,$ $n\geq2,$  is $\mathbb{Z}_2$-contractible to its fixed point $\mathbb{B}.$
\end{theorem}
\begin{proof}
 Let $g$ be the geometric mean on $\mathcal K^n_{(0), b}\times \mathcal K^n_{(0), b} $  (see equation (\ref{eq:geometric mean})), and consider the map $\varphi:\mathcal{K}_0^n\times[0,1]\rightarrow\mathcal{K}_0^n$ defined as
\begin{align}
\label{eq:mainmap}
\varphi(A,t)=\begin{cases}
A\quad\text{if }t=0,\\
g(\psi_t(A),\psi_t(A^\circ)^\circ)\quad\text{if }t\in(0,1),\\
\mathbb{B}\quad\text{if }t=1,
\end{cases}
\end{align}
for every $(A,t)\in\mathcal K^n_0\times[0,1]$.
Observe that $\varphi_0$ is the 
identity on $\mathcal{K}_0^n,$ and $\varphi_1(A)=\mathbb{B}$ for every $A\in \mathcal K^n_0$.
In order to complete the proof, it is enough to verify that $\varphi$ is  continuous and that $\varphi_t(A^\circ)=\varphi_t(A)^\circ$ for every $t\in [0,1]$ and every $A\in\mathcal K^n_0$.
Clearly $\varphi_0(A^\circ)=A^\circ=\varphi_0(A)^\circ$ and $\varphi_1(A^\circ)=\mathbb B=\mathbb B^{\circ}=\varphi_1(A)^\circ$. If $t\in(0,1),$  we can use
properties $(\Gamma1)$ and $(\Gamma2)$ to obtain
\begin{align*}
\varphi_t(A^\circ)&=g\left(\psi_t(A^\circ),\psi_t(A^{\circ\circ})^\circ\right)
=g\left(\psi_t(A^\circ),\psi_t(A)^\circ\right)\\
&=g\left(\psi_t(A)^\circ,\psi_t(A^\circ)\right)
=g(\psi_t(A)^\circ,\psi_t(A^\circ)^{\circ\circ})\\
&=g\left(\psi_t(A),\psi_t(A^\circ)^{\circ}\right)^\circ
=\varphi_t(A)^\circ.
\end{align*}
Hence, $\varphi_t$ is an equivariant map for every $t$. 

To prove that $\varphi$ is continuous, take any pair $(K,s)\in\mathcal K^{n}_0\times [0,1]$.  We  consider separately the cases $s\in(0,1),$ $s=1$ and $s=0.$ The continuity of $\varphi$ on pairs $(K,s),$ $s\in(0,1)$ and $K\in\mathcal{K}_0^n,$  follows from the continuity of both  $\psi$ (Lemma \ref{lem:psimap}) and $\alpha$ (Proposition \ref{prop:polar-continuous}), together with the continuity of the geometric mean  w.r.t.
$d_{AW}$ on $\mathcal{K}_{(0),b}^n.$  The latter is a consequence of the facts that $g$ is continuous on $(\mathcal{K}_{(0),b}^n,d_H)$ (see ($\Gamma5$)), 
and that the metrics $d_{AW}$ and $d_H$ induce the same topology on  $\mathcal{K}_{(0),b}^n$ (Fact \ref{fact:d_H y d_AW coincinden en K_b}). 

If $s=1$, recall that $\psi_t(K)=\psi_t(K^{\circ})=\frac{1-t}{t}\mathbb{B},$  for all $K\in\mathcal{K}_0^n,$ and $t\in \left[ \frac{\sqrt5-1}{2}, 1\right)$   (Lemma \ref{lem:psimap}). Hence, by properties ($\Gamma3$) and (P7), we know that
$$\varphi(K,t)=g\left(\frac{1-t}{t}\mathbb{B},\frac{t}{1-t}\mathbb{B}\right)=\mathbb{B}.$$
Thus,  for every $\delta\in \left [ \frac{\sqrt5-1}{2}, 1\right)$, the set $\mathcal U:=\mathcal K^n_0\times (\delta, 1]$ is an open neighborhood of the pair $(K,1)$ such that $\varphi|_{\mathcal U}$ is constant. Therefore $\varphi$ is continuous on 
$(K,1),$ for every $K\in\mathcal{K}_0^n.$

The case $s=0$ will require the following claims:

\textbf{Claim 1.} For any $t\in\left(0,\frac{1}{2}\right)$ and $A\in\mathcal{K}_0^n$ we have that $$\varphi_t({A})\subseteq A+\frac{t}{1-t}\mathbb{B}.$$

\textit{Proof of Claim 1.}
Let $t\in\left(0,\frac{1}{2}\right)$ and $A\in\mathcal{K}_0^n$ be fixed. Since $\psi_t({A})=(A+t\mathbb{B})\cap\frac{1-t}{t}\mathbb{B},$ then 
$\psi_t({A})\subseteq A+t\mathbb{B}\subseteq A+\frac{t}{1-t}\mathbb{B},$ and 
$\psi_t({A})\subseteq\frac{1-t}{t}\mathbb{B}\subseteq\frac{1}{t}\mathbb{B}. $ So,
\begin{equation}
\label{eq:inclu1claim1}
\psi_t({A})\subseteq\left(A+\frac{t}{1-t}\mathbb{B}\right)\cap\frac{1}{t}\mathbb{B}.
\end{equation}

On the other hand, we know that $A^\circ\subseteq A^\circ+t\mathbb{B},$ hence $(A^\circ+t\mathbb{B})^\circ\subseteq A^{\circ\circ}=A.$ Since  $0$ belongs to both
$(A^\circ+t\mathbb{B})^\circ$ and $\mathbb{B},$ we have that $(A^\circ+t\mathbb{B})^\circ\cup\frac{t}{1-t}\mathbb{B}\subseteq A+\frac{t}{1-t}\mathbb{B}.$ In consequence, we can use  the fact that $A+\frac{t}{1-t}\mathbb{B}$ is closed and convex to infer that
\begin{align}
\label{eq:inclu2claim1}
\psi_t(A^\circ)^\circ&=\left[(A^\circ+t\mathbb{B})\cap\frac{1-t}{t}\mathbb{B}\right]^\circ
\\&=\cco\left[(A^\circ+t\mathbb{B})^\circ\cup\frac{t}{1-t}\mathbb{B}\right]\nonumber
\\&\subseteq A+\frac{t}{1-t}\mathbb{B}.\nonumber
\end{align}
 Since
$0<t<\frac{1}{2},$ then $t\mathbb{B}\subset\frac{1-t}{t}\mathbb{B}.$ This yields to $t\mathbb{B}\subseteq \big(A^\circ+t\mathbb{B}\big)\cap\frac{1-t}{t}\mathbb{B}.$ Therefore, $t\mathbb{B}\subseteq\psi_t(A^\circ)$ and $\psi_t(A^\circ)^\circ\subseteq\frac{1}{t}\mathbb{B}.$ 
This, in combination with (\ref{eq:inclu2claim1}), implies that
\begin{equation*}
\psi_t(A^\circ)^\circ\subseteq\left(A+\frac{t}{1-t}\mathbb{B}\right)\cap\frac{1}{t}\mathbb{B}.
\end{equation*} 

Finally, we can combine the above inclusion together with (\ref{eq:inclu1claim1}), and properties $(\Gamma1)$ and $(\Gamma4)$ to obtain 
\begin{align*}
\varphi_t(A)&=g\left(\psi_t({A}),\psi_t(A^\circ)^\circ\right)\\
&\subseteq g\left(\left(A+\frac{t}{1-t}\mathbb{B}\right)\cap\frac{1}{t}\mathbb{B},\left(A+\frac{t}{1-t}\mathbb{B}\right)\cap\frac{1}{t}\mathbb{B}\right)\\
&=\left(A+\frac{t}{1-t}\mathbb{B}\right)\cap\frac{1}{t}\mathbb{B}\\
&\subseteq A+\frac{t}{1-t}\mathbb{B}.
\end{align*}
\QEDA

\textbf{Claim 2.} For every $\varepsilon>0$, there exists $\eta>0$ such that  if $0<t<\eta,$  then $d_{AW}(\varphi_t(A),A)<\varepsilon$ for all $A\in\mathcal{K}_0^n.$

\textit{Proof of Claim 2.} Let $A\in\mathcal{K}_0^n$ and let us suppose that $\frac{1}{j+1}<\varepsilon\leq\frac{1}{j}$ for some integer $j\geq1.$ Notice that, by Lemma \ref{lem:car-dAW}-(5), $d_{AW}\big(\varphi_t(A),A\big)<\varepsilon$ if and only if  $d_H\big(\varphi_t(A)\cap j\mathbb B,A\cap j\mathbb B\big)<\varepsilon.$ Thus, in order to prove the claim, we will show that  $d_H\big(\varphi_t(A)\cap j\mathbb B,A\cap j\mathbb B\big)<\varepsilon$ if  $t<\eta<\text{min}\{\frac{1}{2},\frac{\varepsilon}{\varepsilon+1},\frac{1}{j+1},\frac{\varepsilon}{j^2}\}.$ Indeed, since $\varphi_t(A)\subseteq A+\frac{t}{1-t}\mathbb{B}$ (Claim 1) and $\frac{t}{1-t}<\varepsilon,$ then $\varphi_t(A)\subseteq A+\varepsilon\mathbb{B}.$ In consequence, for every $x\in\varphi_t(A)\cap j\mathbb B,$ 
$d(x,A\cap j\mathbb B)=d(x,A)<\varepsilon$ (Lemma \ref{lem:car-dAW}-(1)). So $\varphi_t(A)\cap j\mathbb B\subseteq A\cap j\mathbb B+\varepsilon \mathbb{B}.$

To show that $A\cap j\mathbb B\subseteq\varphi_t(A)\cap j\mathbb B+\varepsilon \mathbb{B},$ pick an arbitrary point $a\in A\cap j\mathbb B$ and define $b:=\frac{1}{1+t\|a\|}a$. Notice that $b\in A$ and $\|b\|\leq\|a\|\leq j.$ Moreover, since  $t<\frac{1}{j+1},$ we also have that $\|b\|\leq j<\frac{1-t}{t}.$ Hence,  
$$b\in(A+t\mathbb{B})\cap\frac{1-t}{t}\mathbb{B}=\psi_t(A).$$
 On the other hand, since $b\in\left[0,\frac{1}{1+t\|a\|}a\right]\subseteq\kappa_t(A)$ (where $\kappa_t(A)$ is the set defined in equation~(\ref{eq:kappa})), then $b\in\psi_t(A^\circ)^\circ$ (Lemma \ref{lem:kappa-map}). Thus, if we define
$M:=\psi_t(A)\cap\psi_t(A^\circ)^\circ,$ we can use  properties $(\Gamma1)$ and $(\Gamma4)$  to conclude that $b\in M=g(M,M)\subseteq g(\psi_t(A),\psi_t(A^\circ)^\circ)=\varphi_t(A).$ Hence, $b\in\varphi_t(A)\cap j\mathbb B$ and $\|a-b\|\leq\|a\|^2t<j^2t<\varepsilon$. This proves that $A\cap j\mathbb B\subseteq\varphi_t(A)\cap j\mathbb B+\varepsilon \mathbb{B} $  and therefore $d_H(\varphi_t(A)\cap j\mathbb B,A\cap j\mathbb B)<\varepsilon$, as desired. 
\QEDA

We turn to prove the continuity of $\varphi$  at the point $(K,0)\in \mathcal{K}_0^n\times [0,1]$. To this end, let $\rho>0$ and choose $0<\eta<\rho$ such that 
$d_{AW}(\varphi_t(A),A)<\frac{\rho}{2},$ for every $A\in\mathcal{K}_0^n$ and  $t\in (0,\eta)$  (Claim 2). 
Finally, pick any pair $(A,t)\in\mathcal{K}_0^n\times [0,1]$ with $d_{AW}(K,A)<\frac{\rho}{2}$ and $t<\eta$. If $t=0$, then $\varphi_t(A)=\varphi_0(A)=A$ and therefore
$$d_{AW}(\varphi_0(A), \varphi_0(K))=d_{AW}(A,K)<\frac{\rho}{2}.$$

On the other hand, if $t>0$, by the choice of $\eta$, we conclude that 
$$
d_{AW}(\varphi_t(A),K)\leq d_{AW}(\varphi_t(A),A)+d_{AW}(A,K)<\frac{\rho}{2}+\frac{\rho}{2}=\rho.
$$
Therefore $\varphi$ is continuous at $(K,0)$ and now the proof is completed. 
\end{proof}

Finally, we can combine Theorem~\ref{thm:K0-Z2-contractible} with Remark~\ref{rem: prueba Main 1} to  
obtain the following corollary.

\begin{corollary}
\label{cor:K0-AR}
\begin{enumerate}
    \item $\mathcal{K}^n_0$ is a $\mathbb Z_2$-$\mathrm{AR}$.
    \item $\mathcal{K}^n_0/\alpha$ is an $\mathrm{AR}$.
    \item Theorem~\ref{Theo:main 1} is true. 
\end{enumerate}
\end{corollary}

\section{Polar preserving maps}
\label{sec:polarpreservingmap}

A map $F:\mathcal{K}_0^n\rightarrow\mathcal{K}_0^n$ such that $F(A^\circ)=F(A)^\circ,$ for all 
$A\in\mathcal{K}_0^n,$ is called \textit{polar-equivariant}. In the proof of Theorem \ref{thm:K0-Z2-contractible}, we have constructed a family of polar-equivariant maps 
$\varphi_t:=\varphi(\cdot,t)$ on $\mathcal{K}_0^n,$ for $t\in[0,1].$  As follows from their definition (\ref{eq:mainmap}), the maps $\varphi_t$ are determined by two ``global''  operations:  the polar mapping and the geometric mean. In contrast, there are polar-equivariant maps determined exclusively by point maps on $\mathbb{R}^n.$ This is the case of the orthogonal maps $U\in O(n),$ since the induced map $\widetilde{U}:\mathcal{K}_0^n\rightarrow\mathcal{K}_0^n,$ given by $\widetilde{U}(A):=\{U(a): a\in A\},$ for $A\in\mathcal{K}_0^n,$ is a polar-equivariant map. 

In the following, we investigate the properties of the polar-equivariant maps 
induced by maps $f:\mathbb{R}^n\rightarrow\mathbb{R}^n.$ To do so, we require the following definition.

\begin{definition}
A map $f:\mathbb{R}^n\rightarrow\mathbb{R}^n$ is called polar preserving map if $f(A^\circ)=f(A)^\circ,$ for all $A\in\mathcal{K}_0^n.$ 
\end{definition}

The next properties are direct consequences of the above definition: 
\begin{remark}
\label{rem:obs1}
Let $f:\mathbb{R}^n\rightarrow\mathbb{R}^n$ be a polar preserving map. Then the following hold:
\begin{enumerate}
\item The  induced map $\widetilde{f}:\mathcal{K}_0^n\rightarrow\mathcal{K}_0^n$ given by $\widetilde{f}(A):=f(A),$ for $A\in\mathcal{K}_0^n,$ is a well-defined polar-equivariant map. 
\item $f$ is a surjective map and $f(0)=0$.
\item $f(\mathbb{B})=\mathbb{B}.$
\end{enumerate}
\end{remark}

\begin{proof}
(1) Let $A\in\mathcal{K}_0^n.$ By the Bipolar Theorem (P8), $f(A)=f(A^{\circ\circ})=f(A^\circ)^\circ$. Since the polar set $f(A^\circ)^\circ$ always belongs to $\mathcal{K}_0^n$ (P1), $f(A)\in\mathcal{K}_0^n$ too. In consequence, the map $\widetilde{f}$ is well-defined and clearly satisfies that $\widetilde{f}(A^\circ)=f(A^\circ)=f(A)^\circ= \widetilde{f}(A)^\circ$ for all $A\in\mathcal{K}_0^n.$

(2) By (1) of this remark,  $\widetilde{f}(\{0\})=\{f(0)\}\in\mathcal{K}_0^n.$ Since the only singleton in $\mathcal{K}_0^n$ is $\{0\},$ we conclude that $f(0)=0.$ In addition $f(\mathbb{R}^n)=f(\{0\}^\circ)=\{0\}^\circ=\mathbb{R}^n.$ Hence $f$ is a surjective map. 

(3) By (P5),  $f(\mathbb{B})=f(\mathbb{B}^\circ)=f(\mathbb{B})^\circ, $ and therefore  $f(\mathbb{B})=\mathbb{B}.$  
\end{proof}

\begin{lemma}\label{lem:f=T}
Let $f:\mathbb R^n\to\mathbb R^n$ be a map and $T\in GL(n)$. If $f\big([0,x]\big)=[0, T(x)]$ for every $x\in \mathbb R^n$, then $f=T$.
\end{lemma}
\begin{proof}
Since $T(0)=0$, we obviously have $f(0)=0=T(0)$. Now, if $x\in\mathbb R^n\setminus\{0\}$, we can use the equality 
 $f([0,x])=[0,T(x)]$   to find a  $t\in(0,1]$ such that $f(tx)=T(x).$ Thus, 
$T(x)\in f([0,tx])=[0,tT(x)].$ But, this is only possible  if $t=1$ and therefore $f(x)=T(x)$, as required.
\end{proof}

\begin{lemma}
\label{lem:scalarproductREP}
Let $f:\mathbb{R}^n\to\mathbb{R}^n$ be a  surjective  map such that for every $x,y\in\mathbb{R}^n,$ 
\begin{equation}
\label{eq:scalarineq}
\langle x,y\rangle\leq1\text{ if and only if }\langle f(x),f(y)\rangle\leq1.
\end{equation}  
Then $f$ is a polar preserving map.
\end{lemma}
\begin{proof}
Let $A\in\mathcal{K}_0^n.$ If $x\in A^\circ$, then $\langle a,x\rangle\leq1$ for every $a\in A.$ Hence, by (\ref{eq:scalarineq}), $\langle f(a),f(x)\rangle\leq1$ for every $a\in A.$ Thus, $f(x)\in f(A)^\circ$ and therefore $f(A^\circ)\subseteq f(A)^\circ.$ Conversely, if $y\in f(A)^\circ$, by the surjectivity of $f$,  there exists  $x\in\mathbb{R}^n$ such that $y=f(x)$. In consequence $\langle f(a),f(x)\rangle\leq1$ for all $a\in A$, and by (\ref{eq:scalarineq}), $\langle a,x\rangle\leq1$ for every $a\in A$. This shows that  $x\in A^\circ$ and therefore $y\in f(A^\circ)$. Hence,  $f(A)^\circ\subseteq f(A^\circ)$ and consequently $f(A^\circ)=f(A)^\circ$, as desired.
\end{proof}

Below, we shall see that if  $n\geq2,$ then the family of injective maps on $\mathbb{R}^n$ which are polar preserving consists only of orthogonal maps.  This is not the case if $n=1$, where we can find maps $f:\mathbb R\to \mathbb R$ which are bijective,  polar-preserving but not orthogonal. For example, all the maps of the form  $f_k(x)=x^{2k+1}$
with $k=1,2,\ldots,$ satisfy condition (\ref{eq:scalarineq}) and therefore they are polar-preserving maps.

Recall that for every pair $A,K$ in $\mathcal{K}_0^n$ ($\mathcal{K}_{(0),b}^n$, resp.), the intersection $A\cap K$ and the closed convex hull 
$A\vee K:=\cco(A\cup K)$ belong to $\mathcal{K}_0^n$ 
($\mathcal{K}_{(0),b}^n$, resp.). In fact, these operations endow $\mathcal{K}_0^n$  and $\mathcal{K}_{(0),b}^n$  with a natural lattice structure (see e.g. \cite{SchneiderBoroczky2008,Slomka2011}).

In the main theorem of \cite{SchneiderBoroczky2008}  and \cite[Theorem 2]{Slomka2011} all endomorphisms of the lattices $\big(\mathcal{K}_{(0),b}^n,\cap,\vee\big)$ and $\big(\mathcal{K}_{0}^n,\cap,\vee\big)$ are characterized, respectively. These theorems are summarized in the following fact that will be used through the section. 

\begin{fact}\label{fact:teoremas-Slomka2011-SchneiderBoroczky2008}
Let $f:\mathcal{K}_{0}^n\rightarrow\mathcal{K}_{0}^n$  
($f:\mathcal{K}_{(0),b}^n\rightarrow\mathcal{K}_{(0),b}^n$, resp.) be a mapping satisfying
\begin{equation} \label{eq:lattice-structure-preserver}
    f(A\cap K)=f(A)\cap f(K)\text{ and }
    f(A\vee K)=f(A)\vee f(K),
\end{equation}
for all $A,K\in\mathcal{K}_{0}^n$ 
($A,K\in\mathcal{K}_{(0),b}^n$, resp.). Then, either $f$ is constant or there exists a linear map $T\in GL(n)$ such that $f(A)=T(A)$ for all $A\in\mathcal{K}_{0}^n$ (for all $A\in\mathcal{K}_{(0),b}^n$, resp.).
\end{fact}
\begin{proposition}
\label{prop:injective-polarpreservingmap}
Let $n\geq 2$. If $f:\mathbb{R}^n\rightarrow\mathbb{R}^n$  is a polar preserving  injective map,  then $f$ is an orthogonal map.
\end{proposition}
\begin{proof}
By Remark \ref{rem:obs1}-(1), the induced map $\widetilde{f}:\mathcal{K}_0^n\rightarrow\mathcal{K}_0^n$ is a polar-equivariant map. Even more, $\widetilde{f}$ is not constant, since $\widetilde{f}(\{0\})=\{0\}$ and 
$\widetilde{f}(\mathbb{B})=\mathbb{B}$ (Remark \ref{rem:obs1}-(2)-(3)). 
In addition, by the injectivity of $f,$ $\widetilde{f}(A\cap K)=\widetilde{f}(A)\cap\widetilde{f}(K)$ for every $A,K\in\mathcal{K}_0^n.$  Hence, by properties (P3), (P4) and (P8),
\begin{align*}
    \widetilde{f}(A\vee K)&=\widetilde{f}\big((A\vee K)^{\circ\circ}\big)=\widetilde{f}\big((A^\circ\cap K^\circ)^\circ\big)\\
    &=\widetilde{f}(A^\circ\cap K^\circ)^\circ=\big(\widetilde{f}(A^\circ)\cap\widetilde{f}(K^\circ)\big)^\circ\\
&=\widetilde{f}(A^\circ)^\circ\vee\widetilde{f}(K^\circ)^\circ=\widetilde{f}(A^{\circ\circ})\vee \widetilde{f}(K^{\circ\circ})\\
&=\widetilde{f}(A)\vee\widetilde{f}(K).
\end{align*}

In consequence, the map $\widetilde{f}$ is a non-constant endomorphism of the lattice $(\mathcal{K}_0^n,\cap,\vee)$ and, by Fact \ref{fact:teoremas-Slomka2011-SchneiderBoroczky2008}, there exists a linear isomorphism $T:\mathbb{R}^n\rightarrow\mathbb{R}^n$ such that $\widetilde{f}(A)=T(A)$ for all $A\in\mathcal{K}_0^n$. In particular, $f([0,x])=\widetilde{f}([0,x])=T([0,x])=[0,T(x)]$ for every $x\in \mathbb R^n$, and therefore $f=T$ (Lemma~\ref{lem:f=T}). Finally,  by Remark \ref{rem:obs1}-(3), $T(\mathbb{B})=\widetilde{f}(\mathbb{B})=\mathbb{B},$ and hence $T=f$ is an orthogonal map.
\end{proof}

A direct application of Proposition \ref{prop:injective-polarpreservingmap} and Lemma \ref{lem:scalarproductREP} leads to the following characterization of the bijective maps satisfying condition (\ref{eq:scalarineq}).
\begin{remark}
Let  $f:\mathbb{R}^n\rightarrow\mathbb{R}^n,$ $n\geq 2,$ be a bijective map such that (\ref{eq:scalarineq}) holds. Then $f$ is an orthogonal map.
\end{remark}
We have been interested in polar-equivariant maps $\widetilde{f}$ on $\mathcal{K}_0^n$ induced by maps $f$ on $\mathbb{R}^n.$ However, one can consider a wider class of maps by allowing to take closures. The next proposition examines the properties of the maps $f$ on $\mathbb{R}^n$ such that $\overline{f(A^\circ)}=\overline{f(A)}^\circ,$ for all $A\in\mathcal{K}_0^n.$
\begin{proposition}
\label{prop:continuous-injective-polar-preserving}
Let  $f:\mathbb{R}^n\rightarrow\mathbb{R}^n,$ $n\geq2,$ be a continuous injective map such that 
$\overline{f(A^\circ)}=\overline{f(A)}^\circ,$ for all $A\in\mathcal{K}_0^n.$ Then $f$ is an orthogonal map.
\end{proposition}
\begin{proof}
Let us define $F(A):=\overline{f(A)},$ for $A\in\mathcal{K}_0^n.$ We will show  that the restriction $F|_{\mathcal{K}_{(0),b}^n}:\mathcal{K}_{(0),b}^n\rightarrow\mathcal{K}_{(0),b}^n$  is a well-defined non-constant map satisfying the hypotheses of Fact \ref{fact:teoremas-Slomka2011-SchneiderBoroczky2008}.
This would guarantee the existence of a linear isomorphism $T:\mathbb{R}^n\rightarrow\mathbb{R}^n$ such that $F(K)=T(K)$ for all $K\in\mathcal{K}_{(0),b}^n$.  Our work then will consist in proving that $f=T$ and that $T$ is an orthogonal map.

Notice that 
$F(A)=\overline{f(A^{\circ\circ})}=\overline{f(A^{\circ})}^\circ=F(A^\circ)^\circ,$ for all $A\in\mathcal{K}_0^n.$  Hence $F(A)\in\mathcal{K}_0^n$ and $F(A^\circ)=F(A)^\circ$ for all $A\in\mathcal{K}_0^n.$ It follows directly that $f(0)=0,$ $\overline{f(\mathbb{R}^n)}=\mathbb{R}^n$ and $\overline{f(\mathbb{B})}=\mathbb{B}.$ Moreover, by the continuity of 
$f$, $f(K)$ is compact for all $K\in\mathcal{K}_{(0),b}^n$. Thus, $F(K)=\overline{f(K)}=f(K)$ for each $K\in\mathcal{K}_{(0),b}^n$.  Furthermore,  if $K\in \mathcal K^n_{(0),b}$,  $K^\circ$ also belongs to $\mathcal{K}_{(0),b}^n$  (property (P1)) and then $f(K^\circ)=f(K)^\circ$ is compact.
From this, we infer that $f(K)\in\mathcal{K}_{(0),b}^n$ for all $K\in\mathcal{K}_{(0),b}^n.$ 
The previous observations show that $F|_{\mathcal{K}_{(0),b}^n}:\mathcal{K}_{(0),b}^n\rightarrow\mathcal{K}_{(0),b}^n$  is well-defined, $F(\mathbb{B})=\mathbb{B},$ and $F(K)=f(K)$ for every $K\in\mathcal K^n_{(0), b}$.   

We can now proceed as in the proof of Proposition \ref{prop:injective-polarpreservingmap} to show that $F(A\cap K)=F(A)\cap F(K)$ and $F(A\vee K)=F(A)\vee F(K)$ for all $A,K\in\mathcal{K}_{(0),b}^n.$ Hence, by Fact \ref{fact:teoremas-Slomka2011-SchneiderBoroczky2008}, $F(K)=T(K)$ for some linear isomorphism $T$ on $\mathbb{R}^n$ and for all 
$K\in\mathcal{K}_{(0),b}^n$. Clearly, $T$ must be an orthogonal map, since $F(\mathbb{B})=\mathbb{B}.$
To prove that $f=T$, let  $x\in\mathbb{R}^n$ and define, for each integer $m\geq1$,  the set 
$$X_m:=[0,x]+\frac{1}{m}\mathbb{B}.$$
Then $f(X_m)=F(X_m)=[0,T(x)]+\frac{1}{m}\mathbb{B}.$ Hence,  $(f(X_m))_m$ converges (w.r.t. $d_H$) to $[0,T(x)].$ Since $(X_m)_m$ converges (w.r.t. $d_H$) to $[0,x],$ and  $f$          is a continuous map, we have that $f([0,x])=[0,T(x)]$ for all  $x\in\mathbb{R}^n$.
We can now use Lemma~\ref{lem:f=T} to conclude that $f=T$, as required. 
\end{proof}

\section{Decreasing involutions on $\mathcal{K}^n_0$}
\label{sec:other standard involutions}

In Theorem~A
(\cite[Corollary 4]{Slomka2011}), the maps $f:\mathcal{K}_0^n\rightarrow\mathcal{K}_0^n$ satisfying conditions (D1) and (D2) are completely described. 
It is shown that such maps $f$ are of the form  $f(A)=T(A^\circ)$ for some symmetric linear isomorphism $T:\mathbb{R}^n\rightarrow\mathbb{R}^n$ and all $A\in\mathcal{K}_0^n$. 
Bearing in mind Theorem~\ref{Theo:main 1}, it is then natural to ask when does a map $f:\mathcal{K}_0^n\rightarrow\mathcal{K}_0^n$ satisfying (D1) and (D2) is conjugate with the polar mapping (and therefore  with the standard involution on $Q$). Theorem \ref{thm:caracterizacion-decreasing}  answers this question, and the purpose of this section is to prove it.

The following remark will be used in the proof of Theorem \ref{thm:caracterizacion-decreasing}. It establishes the continuity of the maps  on $\mathcal{K}_0^n$ induced by linear isomorphisms on $\mathbb{R}^n$.

\begin{remark}\label{rem:continuity-decreasing-inv}
Let $T:\mathbb{R}^n\rightarrow\mathbb{R}^n$ be a linear isomorphism. Then, the maps $\widetilde{T},f:\mathcal{K}_0^n\rightarrow\mathcal{K}_0^n$, given  by $\widetilde{T}(A):=T(A)$ and  $f(A):=T(A^\circ)$, are continuous.
\end{remark}
\begin{proof}
Using that $T$ is a continuos bijection, we can easily check that for every open set $U\subset \mathbb R^n$ and every compact set $C\subset \mathbb R^n$, the following equalities hold:
$$\widetilde T^{-1}(U^{-})=(T^{-1}(U))^{-}\;\text{ and }\;\widetilde T^{-1}((\mathbb{R}^n\setminus C)^{+})= (\mathbb R^{n}\setminus T^{-1}(C))^+.$$
Thus, $\widetilde T^{-1}(U^{-})$ and $\widetilde T^{-1}((\mathbb{R}^n\setminus C)^{+})$ are open sets with respect to the Fell topology. The continuity of $\widetilde T$ then follows immediately from Fact \ref{fact:topologies are the same}. 
On the other hand, observe that $f$ is such that $f=\widetilde{T}\alpha$. Therefore, its continuity follows from the continuity of $\widetilde{T}$ and the continuity of the polar mapping (Fact \ref{fact: continuidad polar daw }).
\end{proof}

We denote by $M_{n}(\mathbb{R})$ the vector space of square-matrices $R=(r_{ij}),$ $i,j=1,\ldots,n,$ with real entries.  As usual, $GL(n)$ and $O(n)$ denote the set of linear isomorphisms on $\mathbb{R}^n$ and the set of orthogonal maps on $\mathbb{R}^n$, respectively. In what follows, we will not distinguish between a linear map on $\mathbb{R}^n$ and its canonical matrix representation.
For every $T\in M_{n}(\mathbb{R})$, we denote by  $T^\top$ the transpose of $T$. Notice that in the case we are working on, $T^\top$ coincides with the adjoint operator of $T$. Hence,  for every linear isomorphism $T:\mathbb{R}^n\rightarrow\mathbb{R}^n$ and every $A\in\mathcal{K}_0^n,$ the following formula holds:
$$T(A^\circ)=[(T^\top)^{-1}A]^\circ$$
(see, e.g. \cite[Chaper IV \S 2]{SchaeferWolff1999}).

Finally, we are in condition to prove Theorem~\ref{thm:caracterizacion-decreasing}.

\begin{proof}
[Proof of Theorem~\ref{thm:caracterizacion-decreasing}]
First, observe that from Remark \ref{rem:continuity-decreasing-inv}, the map $f:\mathcal{K}_0^n\rightarrow\mathcal{K}_0^n$ is continuous. In particular, it is an involution on $\mathcal{K}_0^n$. 

Let us fix an orthogonal diagonalization of $T$ by some $U\in O(n),$ and a diagonal matrix $D\in M_{n}(\mathbb{R})$. Then $T=U^\top DU$ and $f(A)=U^\top DU(A^\circ)$ for all $A\in\mathcal{K}_0^n.$  

(1) We shall show that if  $T$  is  positive-definite, then  $f$ is conjugate with the polar mapping and, in particular, it is a based-free involution. Indeed, in this case, all  diagonal entries $d_{ii}$, $i=1,\ldots,n,$ of $D$ are strictly positive. Hence, the diagonal matrix $R$ with entries $r_{ii}=\sqrt {d_{ii}}$, $i=1,\ldots,n,$ is well-defined and $D=RR$. Furthermore, 
for every $A\in\mathcal{K}_0^n,$
\begin{align*}
f(A)=U^\top &RRU(A^\circ)=U^\top R\left[\left(((RU)^\top)^{-1}A\right)^\circ\right]\\
&=U^\top R\left[((U^\top R)^{-1}A)^\circ\right].
\end{align*}
By letting $\Psi=U^\top R,$ we have that $f(A)=\Psi((\Psi^{-1} A)^\circ)$ for all $A$. The latter shows that 
$f=\widetilde \Psi\alpha\widetilde{\Psi^{-1}}$, where $\widetilde \Psi:\mathcal K^n_0\to \mathcal K^n_0$ and $\widetilde{\Psi^{-1}}:\mathcal K^n_0\to \mathcal K^n_0$ are the maps induced by $\Psi$ and $\Psi^{-1}$, respectively. Clearly,
the maps  $\widetilde \Psi$ and  $\widetilde{\Psi^{-1}}$ are continuous (Remark \ref{rem:continuity-decreasing-inv}), and $\widetilde \Psi^{-1}=\widetilde{\Psi^{-1}}$.  In consequence,  $\widetilde \Psi$ is a homeomorphism and therefore $f$ is conjugate with the polar mapping $\alpha$, as desired.


 
(2) Let us suppose that $T$ is not positive-definite. We shall show that $f$ has infinitely many fixed points lying in $\mathcal{K}_{(0),b}^n$.  
Let $S\in M_{n}(\mathbb{R})$ be the diagonal matrix such that $s_{ii}=\sqrt{|d_{ii}|}$ for $i=1,\ldots,n.$ Notice that $W=S^{-1}DS^{-1}$ is a diagonal matrix such that  $\varepsilon_{i}:=w_{ii}\in\{-1,1\}$ for $i=1,\dots,n$. Moreover, 
$\varepsilon_{i}=w_{ii}=-1$ iff $d_{ii}<0$.
Clearly, $W$ is an orthogonal matrix and $W^{-1}=W$.

To exhibit an infinite family of fixed points, let us fix an index $j\in \{1,\dots, n\}$ such that  $d_{jj}<0$,  and define the set
\begin{equation}
\label{eq:cono-Aj}
A_{j}:=\left\{(a_1,\ldots,a_n)\in\mathbb{R}^n: a_j\geq\sqrt{\sum_{i\neq j}a_i^2}\right\}.
\end{equation}
The following is a well-known fact about the cone $A_j$. However, we include its proof for the sake of completeness:

\noindent\textbf{Claim.} $W(A_j)=A_j^\circ.$  

\textit{Proof of the Claim.} Let $a=(a_1,\ldots,a_n)$ and $x=(x_1,\ldots,x_n)$ be points in $A_j.$ Then, by the Cauchy-Schwarz inequality,
\begin{equation*}
\langle W(a),x \rangle=\left(\sum_{i\neq j}\varepsilon_{i}a_ix_i\right)-a_jx_j
\leq
\left(\sqrt{\sum_{i\neq j}(\varepsilon_{i}a_i)^2}\right)\left(\sqrt{\sum_{i\neq j}x_i^2}\right)-a_jx_j
\end{equation*}
Hence, $\langle W(a),x \rangle\leq 0$ and therefore $W(a)\in A_j^\circ.$ This proves that $W(A_j)\subseteq A_j^\circ.$ 

On the other hand, if $y=(y_1,\ldots,y_n)\in A_j^\circ,$ then for every integer $m\geq1$ and every $a=(a_1,\ldots,a_n)\in A_j$ such that $a_j=\sqrt{\sum_{i\neq j}a_i^2}=m,$ we have that
\begin{align*}
\langle y,a\rangle&=\left(\sum_{i\neq j}y_ia_i\right)+y_jm\leq 1\;\;\;\text{ and therefore}\\
&\frac{1}{m}\left(\sum_{i\neq j}y_ia_i\right)\leq \frac{1}{m}-y_j.
\end{align*}  
Since $m$ is an arbitrary positive integer, and the supremun of $\frac{1}{m}\sum_{i\neq j}y_ia_i$ over the $(n-1)$-tuples
$(a_1,\ldots,a_{j-1},a_{j+1},\ldots,a_n)$ such that $\sqrt{\sum_{i\neq j}a_i^2}=m$ is 
$\sqrt{\sum_{i\neq j}y_i^2},$  
we conclude that 
$\sqrt{\sum_{i\neq j}y_i^2}\leq-y_j.$ As a consequence, the point
$$y_0:=W(y)=(\varepsilon_{1}y_{1},\ldots,\varepsilon_{j-1}y_{j-1},-y_j,\varepsilon_{j+1}y_{j+1},\ldots,\varepsilon_{n}y_{n})$$ belongs to $A_j$.  It then follows that $y=W(y_0)\in W(A_j).$ We infer that $A_j^\circ\subseteq W(A_j)$  and therefore $A_j^\circ=W(A_j)$, as required.
\QEDA

Since $W^{-1}=W,$ then $A_j=W(A_j^\circ).$ This, in combination with $(P1)$, implies that $A_j\in \mathcal K^n_0$.

Now, let $E_t,K_t\in\mathcal{K}_{(0),b}^{n},$ with $t\in(0,1),$ be defined as 
\begin{align*}
E_t:=(A_j+t\mathbb{B})\cap\left(\frac{1-t}{t}\right)\mathbb{B},\text{ and }
K_t:=\left((A_j^\circ+t\mathbb{B})\cap\left(\frac{1-t}{t}\right)\mathbb{B}\right)^\circ.
\end{align*}
Using the claim, it is not difficult to see that $K_t^\circ=W(E_t)$,  and thus $K_t=W\big(E_t^\circ\big)$. Moreover, the geometric mean $P_t:=g(E_t,K_t)$ belongs to  $\mathcal{K}_{(0),b}^{n}$ and, by properties ($\Gamma1$), ($\Gamma2$) and ($\Gamma6$), we have that
\begin{align*}
P_t^\circ=g(E_t,K_t)^\circ=g\big(E_t^\circ,W(E_t)\big)=g\big(W(K_t),W(E_t)\big)=W(P_t).
\end{align*}

We have just constructed a family of solutions for the equation $X^\circ=W(X)$ on $\mathcal{K}_{(0),b}^{n}$.
Even more, since $E_t=\psi_t(A_j)$, $K_t=(\psi_t(A_j^\circ))^\circ$ and $P_t=\varphi_t(A_j)$ 
(where $\psi$ and $\varphi$ are the maps defined in (\ref{eq:funcion psi}) and (\ref{eq:mainmap}), resp.), then  $P_t\to A_j$ as $t$ tends to $0$.
 Since $A_j$ is an unbounded set, the family $\big\{P_t:t\in(0,1)\big\}\subset\mathcal{K}_{(0),b}^{n}$ must be infinite. 

To finish the proof, let $Y_t$ denote the set $U^\top S(P_t)$, and notice that  $Y_t^\circ=\big(U^\top S(P_t)\big)^\circ=[(U^\top S)^\top ]^{-1}(P_t^\circ)=U^\top S^{-1}(P_t^\circ).$ Therefore,
\begin{align*}
f(Y_t)=U^\top DU(Y_t^\circ)&=U^\top DUU^\top S^{-1}(P_t^\circ)=U^\top DS^{-1}(P_t^\circ)\\
&=U^\top SW(P_t^\circ).
\end{align*}
Since $P_t=W(P_t^\circ),$ then  $f(Y_t)=U^\top S(P_t)=Y_t$ is a fixed point of $f$. Hence, the family $\big\{Y_t:t\in(0,1)\big\}\subset\mathcal{K}_{(0),b}^{n}$ consists only of fixed points of $f$. In particular,
$f$ has infinite fixed points on $\mathcal{K}_{(0),b}^{n}$. 
\end{proof}

\section{Final remarks and questions}
\label{sec:remarks}

In \cite{Rotem2016-Algebraic},  algebraically inspired results for the class $\mathcal{K}_{(0),b}^{n}$ were obtained.  There, by letting the polar set $A^\circ$ play the role of the
multiplicative inverse of $A$, different real-equations can be translated to the context of compact convex sets $\mathcal{K}_{(0),b}^{n}.$ 
In this sense, the equation $X^\circ=-X$ arises from the equation $x^{-1}=-x$ which has no solution over the real numbers. In contrast, as a consequence of Theorem \ref{thm:caracterizacion-decreasing}-(2),  the equation $X^\circ=-X$ has infinite solutions on $\mathcal{K}_{(0),b}^{n}$. This fact follows directly from the remark below. 

\begin{remark}
\label{rem:infinite-fixed points}
Suppose that $T\in GL(n)$ is symmetric and is not positive-definite, then the equation $X^\circ=T(X)$ has infinite solutions on $\mathcal{K}_{(0),b}^{n}.$
\end{remark}
The Hilbert cube $Q=\prod_{i=1}^{\infty}[-1,1]$ has a natural lattice structure. In fact, for all $x=(x_i)_{i\in\mathbb{N}}, y=(y_i)_{i\in\mathbb{N}}\in Q,$ the relation $\preceq,$ defined by $x\preceq y$ iff $x_i\leq y_i$ for all $i\in\mathbb{N}$ is a partial order for which the operations $\vee$ and  $\wedge,$ defined as $x\vee y:=(\max\{x_i,y_i\})_{i}$ and $x\wedge y:=(\min\{x_i,y_i\})_{i},$ for $x,y\in Q,$ provide a lattice structure. In this context, the standard based-free involution $\sigma$ satisfies, additionally, a condition similar to (D2): for $x,y\in Q$, $\sigma(y)\preceq\sigma(x)$ whenever $x\preceq y.$ Namely, it is decreasing with respect to $\preceq$. Taking into account that in Theorem~\ref{thm:caracterizacion-decreasing} we showed that every based-free decreasing involution on $\mathcal{K}_0^n$ 
is conjugate with the polar mapping, we propose the following weaker versions of Anderson's problem:

\textbf{Question 1}. Is every based-free decreasing involution $\beta:Q\rightarrow Q$ conjugate with $\sigma$?

\textbf{Question 2}. Let $\beta:Q\rightarrow Q$ be a based-free involution  with the property that there exist a lattice structure $(Q,\vee',\wedge')$ with respect to a partial order
$\preceq'$ on $Q$, such that $\beta(y)\preceq'\beta(x)$ whenever
$x\preceq' y$. Is $\beta$ conjugate to $\sigma$?

Notice that the converse of Question 2 is trivially true.

Recall that a homeomorphism $\Phi:Q\to\mathcal K^n_0$ is equivariant (with respect to $\sigma$ and $\alpha$) if $\Phi(-x)=\Phi(x)^\circ$ for every $x\in Q$. Since Theorem~\ref{Theo:main 1} guarantees  the existence of at least one equivariant homemorphism, it is then natural to ask if there exist an equivariant homeomorphism $\Phi:Q\to\mathcal K^n_0$ such that $x\preceq y$ if and only if $\Phi(x)\subset \Phi(y)$. 
 However, the answer to this question is negative, as we explain in the following remark.

\begin{remark}\label{rem: order}
There is no homeomorphism $\Phi:Q\to \mathcal  K^n_0$ such that $x\preceq y$ if and only if $\Phi(x)\subset \Phi(y)$. Indeed, if such a homeomorphism exists, then it is not difficult to show that it would satisfy the following two equalities
$$\Phi(x\vee y)=\Phi(x)\vee \phi(y)\quad\text{ and }\quad\Phi(x\wedge y)= \Phi(x)\cap\Phi (y).$$
In particular, the following commutative diagram would hold:

\[
\xymatrix{ Q\times Q  \ar[r]^\vee &Q\ar[d]^\Phi\\ \mathcal{K}^n_0\times\mathcal K^n_0 \ar[u]^{\Phi^{-1}\times\Phi^{-1}}\ar[r]^\vee &\mathcal K_0^n}
\]

Since the operation $\vee$  is continuous on $Q$,  we would have that the operation $\vee$  on $\mathcal K_0^n$ is a composition of continuous maps, and therefore it must be continuous too. However, it is not difficult to see that $\vee$ is not a continuous operation on $\mathcal K_0^n$ and therefore we have a contradiction. 
\end{remark}


Finally, we want to point out another algebraic similitude between $Q$ and $\mathcal K^n_0$. For every $x=(x_i)_{i\in\mathbb{N}}$, define $\frac{1}{2}x:=(\frac{1}{2}x_i)_{i\in\mathbb N}$. Clearly  $\frac{1}{2}x\in Q$ for every $x\in Q$.
If we consider the map $\gamma:Q\times Q\to Q$ given by $\gamma(x,y):=\frac{1}{2}x+\frac{1}{2}y$, then $\gamma$ is well-defined and satisfies the following properties.

\begin{itemize}
\item[($\widetilde\Gamma1$)] $\gamma(x,x)=x$ and $\gamma(x,y)=\gamma(y,x).$   
\item[($\widetilde\Gamma2$)] $\gamma(-x,-y)=-\gamma(x,y).$
\item[($\widetilde\Gamma3$)] $\gamma(x,-x)=0$ (where $0$ is the element of $Q$ with all its coordinates equal to $0$). 
\item[($\widetilde\Gamma4$)]  If $x_1\preceq x_2$ and $y_1\preceq y_2$, then $\gamma(x_1,y_1)\preceq \gamma(x_2,y_2)$ .
\item[($\widetilde\Gamma5$)] The map $\gamma$ is continuous.
\end{itemize}

The reader can notice that if we replace the role of the involution $\sigma$ by $\alpha$, and the order $\preceq$ by the inclusion $\subseteq$, then the above properties are analogous to properties $(\Gamma1)$-$(\Gamma5)$ of the geometric mean $g$. However there is a significant difference: the map $\gamma$ is defined on the the whole space $Q$, while $g$ is only defined on $\mathcal K^n_{(0),b}$.

\begin{remark}\label{rem: geometric mean extended}
For any equivariant homeomorphism $\Phi:Q\to\mathcal K^n_0$, we can define the map $g':\mathcal K^n_0\times\mathcal K^n_0\to\mathcal K^n_0$ given by
$$g'(A,K)=\Phi\Big(\gamma \big(\Phi^{-1}(A), \Phi^{-1}(K)\big)\Big).$$

\begin{enumerate}
    \item The map $g'$ satisfies properties $\Gamma 1$, $\Gamma 2$, $\Gamma 3$.
    \item The map $g'$ also satisfies the following property
    $$(\Gamma' 5)\;\;\text {The map }g'\text { is continuous w.r.t. }d_{AW}.$$
    \end{enumerate}

\end{remark}

After Remark~\ref{rem: geometric mean extended}, we believe that the next final question could be of interest.

\textbf{Question 4.} Is it possible to construct explicitly a map $g':\mathcal K^n_0\times\mathcal K^n_0\to\mathcal K^n_0$ satisfying properties $(\Gamma 1), (\Gamma 2), (\Gamma 3)$ and $(\Gamma' 5)$? And $\Gamma 4$?
 
\subsection*{Acknowledgments}
We wish to  thank  the anonymous  referee for the constructive comments and recommendations
which improved the final version of this paper.

\bibliographystyle{ijmart}

\begin{thebibliography}{10}


\bibitem{Antonyan1990}
{\sc {S. A.} {A}ntonyan.}
\newblock {R}etraction properties of the orbit space (in Russian),
\newblock {\em Mat. Sb. 137}, 3 (1988), 300-318.
\newblock English transl.: {\em}
S. A. Antonyan, (3) (1988) 300–318, 
Math. USSR Sb. 65 (2) (1990) 305–321.

\bibitem{Antonyan1999}
{\sc {S. A}. {A}ntonyan.}
\newblock {O}n based-free actions of compact {L}ie groups on the {H}ilbert
  cube.
\newblock {\em Math. Notes 65}, 2 (1999), 135--143.

\bibitem{Antonyan2005}
{\sc {S. A}. {A}ntonyan.}
\newblock {A} characterization of equivariant absolute extensors and the
  equivariant {D}ugundji theorem.
\newblock {\em {H}ouston J. Math. 31}, 2 (2005), 451--462.

\bibitem{AntonyanNatalia}
{\sc {S. A}. {A}ntonyan, and {N}. {J}onard{-}P{\'e}rez.}
\newblock Affine group acting on hyperspaces of compact convex subsets of
  {$\mathbb{R}^n$}.
\newblock {\em {F}und. {M}ath. 223\/} (2013), 99--136.

\bibitem{ANTONYAN2021}
{\sc {S. A}. {A}ntonyan.}
\newblock {S}ome open problems in equivariant infinite-dimensional topology.
\newblock {\em {T}opology {A}ppl.\/} (2022), 107966.

\bibitem{ArtsteinMilman2007}
{\sc {S.} {A}rtstein-{A}vidan, and {V}. {M}ilman.}
\newblock {A} characterization of the concept of duality.
\newblock {\em Electron. Res. Announc. Math. Sci. 14\/} (2007), 42--59.

\bibitem{ArtsteinMilman2008}
{\sc {S.} {A}rtstein-{A}vidan, and {V}. {M}ilman..}
\newblock {T}he concept of duality for measure projections of convex bodies.
\newblock {\em J. Funct. Anal. 254}, 10 (2008), 2648--2666.

\bibitem{Barvinok}
{\sc {A}. {B}arvinok.}
\newblock {\em {A} {C}ourse in {C}onvexity}, vol.~54.
\newblock {G}raduate {S}tudies in {M}athematics, AMS, 1963.

\bibitem{Beer1993}
{\sc {G}. {B}eer.}
\newblock {\em {T}opologies on {C}losed and {C}losed {C}onvex {S}ets}.
\newblock {K}luwer {A}cademic {P}ublishers {G}roup, {D}ordrecht, 1993.

\bibitem{Berstein-West}
{\sc {I.} {B}erstein, and  {J. E.} {W}est.}
\newblock {{B}ased free compact Lie group actions on {H}ilbert cubes}.
\newblock {\em Proc. Sympos. Pure Math. Algebraic and Geometric Topology,32, Pt. 1}  (1978) 373–391.


\bibitem{SchneiderBoroczky2008}
{\sc {K. J}. {B}\"{o}r\"{o}czky, and {R.} {S}chneider.}
\newblock {A} characterization of the duality mapping for convex bodies.
\newblock {\em Geom. Funct. Anal. 18}, 3 (2008), 657--667.

\bibitem{Borsuk}
{\sc {K.} {B}orsuk.}
\newblock {\em Theory of {R}etracts}.
\newblock {P}olish {S}cientifics {P}ublishers, {W}arsaw, 1966.

\bibitem{Bredon}
{\sc {G. E.} {B}redon.}
\newblock {\em {I}ntroduction to {C}ompact {T}ransformation {G}roups}.
\newblock {A}cademic {P}ress, {N}ew {Y}ork-{L}ondon, 1972.

\bibitem{Chapmanbook}
{\sc {T. A.} {C}hapman.}
\newblock {\em Lectures on {H}ilbert {C}ube {M}anifolds}.
\newblock {A}merican {M}athematical {S}ociety, {P}rovidence, {R. I.}, 1976.
\newblock Expository lectures from the {CBMS} {R}egional {C}onference held at
  {G}uilford {C}ollege, {O}ctober 11-15, 1975, {R}egional {C}onference {S}eries
  in {M}athematics, No. 28.

\bibitem{CurtisSchori74}
{\sc {D.~W.} {C}urtis, and  {R.~M.} {S}chori.}
\newblock {$2^{X}$} and {$C(X)$} are homeomorphic to the {H}ilbert cube.
\newblock {\em {B}ull. {A}mer. {M}ath. {S}oc. 80\/} (1974), 927--931.

\bibitem{CurtisSchori78}
{\sc {D.~W.} {C}urtis, and  {R.~M.} {S}chori.}
\newblock Hyperspaces of {P}eano continua are {H}ilbert cubes.
\newblock {\em {F}und. {M}ath. 101}, 1 (1978), 19--38.

\bibitem{DonJuanJonardPerezLopezPoo}
{\sc {V}. {D}onjuan, {N.} {J}onard{-}{P}{\'e}rez, and {A.} {L}{\'o}pez{-}{P}oo.}
\newblock {S}ome notes on induced functions and group actions on hyperspaces.
\newblock {\em {T}opology {A}ppl.\/} (2022), 107954.

\bibitem{Hu}
{\sc {S. T}. {H}u.}
\newblock {\em {T}heory of {R}etracts}.
\newblock {W}ayne {S}tate {U}niversity {P}ress, {D}etroit, 1965.

\bibitem{MilmanRotem2017}
{\sc {V}. {M}ilman, and {L}. {R}otem.}
\newblock {N}on-standard constructions in convex geometry: geometric means of
  convex bodies.
\newblock In {\em Convexity and concentration}, vol.~161. Springer, New York,
  2017, pp.~361--390.
  
\bibitem{vanMillbook}
{\sc {J}. van {M}ill.}
\newblock {\em Infinite-{D}imensional {T}opology}.
\newblock North-Holland Publishing Co., Amsterdam, 1989.
\newblock Prerequisites and introduction.

\bibitem{vanMillWest2020}
{\sc {J}. van {M}ill, and {J. E}. {W}est.}
\newblock Involutions of $\ell_2$ and $s$ with unique fixed points.
\newblock {\em Trans. Amer. Math. Soc. 373\/} (2020), 7327--7346.

\bibitem{vanMillWest2022}
{\sc {J}. van {M}ill, and {J. E.} {W}est.}
\newblock Universal based-free involutions.
\newblock {\em {T}opology. {A}ppl. 311\/} (2022), 107968.


\bibitem{nadler1979}
{\sc {S}. {N}adler, {J}. {Q}uinn, and {N}. {S}tavrakas.}
\newblock Hyperspaces of compact convex sets.
\newblock {\em {P}acific {J}. {M}ath. 83}, 2 (1979), 441--462.

\bibitem{Rotem2016-Algebraic}
{\sc {L.} {R}otem.}
\newblock {A}lgebraically inspired results on convex functions and bodies.
\newblock {\em Commun. Contemp. Math. 18}, 6 (2016), 1650027, 14.

\bibitem{Rotem2016}
{\sc {L.} {R}otem.}
\newblock {B}anach limit in convexity and geometric means for convex bodies.
\newblock {\em Electron. Res. Announc. Math. Sci. 23\/} (2016), 41--51.

\bibitem{Sakai2013book}
{\sc {K}. {S}akai.}
\newblock {\em {G}eometric {A}spects of {G}eneral {T}opology}.
\newblock Springer, Tokyo, 2013.

\bibitem{SakaiYaguchi2006}
{\sc {K}. {S}akai, and {M.} Yaguchi.}
\newblock The {AR}-property of the spaces of closed convex sets.
\newblock {\em Colloq. Math. 106}, 1 (2006), 15--24.

\bibitem{SakaiYang2007}
{\sc {K}. {S}akai, and {Z.} {Y}ang.}
\newblock {T}he spaces of closed convex sets in {E}uclidean spaces with the
  {F}ell topology.
\newblock {\em Bull. Pol. Acad. Sci. Math. 55}, 2 (2007), 139--143.

\bibitem{SchaeferWolff1999}
{\sc {H.} {S}chaefer, and {M.} {W}olff.}
\newblock {\em {T}opological {V}ector {S}paces}.
\newblock {S}pringer-{V}erlag, {N}ew {Y}ork, 1999.

\bibitem{Schneider1993}
{\sc {R.} {S}chneider.}
\newblock {\em Convex {B}odies: {T}he {B}runn-{M}inkowski {T}heory, volume 44
  of {E}ncyclopedia of {M}athematics and {A}pplications}.
\newblock Cambridge University Press, Cambridge, 1993.

\bibitem{Slomka2011}
{\sc {B.} {S}lomka.}
\newblock {O}n duality and endomorphisms of lattices of closed convex sets.
\newblock {\em {A}dv. {G}eom. 11}, 2 (2011), 225--239.

\bibitem{ThompsonAC}
{\sc {A.~C.} {T}hompson.}
\newblock {\em {M}inkowski {G}eometry}.
\newblock Cambridge University Press, Cambridge, 1996.

\bibitem{Torunczyk1980}
{\sc {H}. {T}oru\'{n}czyk.}
\newblock On {${\rm CE}$}-images of the {H}ilbert cube and characterization of
  {$Q$}-manifolds.
\newblock {\em Fund. Math. 106}, 1 (1980), 31--40.

\bibitem{WebsterR}
{\sc {R}. {W}ebster.}
\newblock {\em {C}onvexity.}
\newblock {O}xford {U}niversiy {P}ress, {N}ew {Y}ork, 1994.

\bibitem{West1976}
{\sc {J. E.} {W}est.}
\newblock {I}nduced involutions on {H}ilbert cube hyperspaces.
\newblock In {\em Topol. {P}roc., {V}ol. {I} ({C}onf., {A}uburn {U}niv.,
  {A}uburn, {A}la., 1976)\/} (1977), pp.~281--293.

\bibitem{West1990}
{\sc {J. E.} {W}est.}
\newblock {O}pen problems in infinite-dimensional topology.
\newblock In {\em {O}pen {P}roblems in {T}opology}. {N}orth-{H}olland,
  {A}msterdam, 1990, pp.~523--597.

\bibitem{west2018}
{\sc {J. E.} {W}est.}
\newblock {I}nvolutions of {H}ilbert cubes that are hyperspaces of {P}eano
  continua.
\newblock {\em {T}opol. {A}ppl. 240\/} (2018), 238--248.

\bibitem{WestWong79}
{\sc {J. E.} {W}est, and {R. Y. T.} {W}ong.}
\newblock Based-free actions of finite groups on {H}ilbert cubes with absolute
  retract orbit spaces are conjugate.
\newblock In {\em {G}eometric {T}opology ({P}roc. {G}eorgia {T}opology {C}onf.,
  {A}thens, {G}a., 1977)\/} (1979), Academic Press, New York-London,
  pp.~655--672.

\bibitem{Wijsman1966}
{\sc {R.~A.} {W}ijsman.}
\newblock Convergence of sequences of convex sets, cones and functions. {II}.
\newblock {\em Trans. Amer. Math. Soc. 123\/} (1966), 32--45.

\bibitem{Wong1974}
{\sc {R. Y. T}. {W}ong.}
\newblock {P}eriodic actions on the {H}ilbert cube.
\newblock {\em Fund. Math. 85\/} (1974), 203--210.

\end{thebibliography}

\end{document}